\documentclass[reqno]{amsart}     

\usepackage{graphicx}
\usepackage{mathtools}
\usepackage{amsfonts}
\usepackage{faktor}
\usepackage{wrapfig}
\usepackage{epstopdf}
\usepackage{caption}
\usepackage{url}
\usepackage{mathtools} 
\usepackage{float}
\usepackage{caption}
\usepackage{centernot}
\usepackage{dsfont}
\usepackage{mathtools}
\DeclarePairedDelimiter{\ceil}{\lceil}{\rceil}
\DeclarePairedDelimiter{\floor}{\lfloor}{\rfloor}
\usepackage{tikz}
\usetikzlibrary{matrix,arrows,decorations.pathmorphing}
\usepackage{pgf}
\usepackage{tikz-cd}
\usepackage{listofitems}
\usepackage{stackengine}

\title[The mapping cone formula and surgery on knots in $S^3$]{The mapping cone formula in Heegaard Floer homology and Dehn surgery on knots in $S^3$}

\author{Fyodor Gainullin}
\email{fyodor.gainullin@gmail.com}

\newcommand{\cA}{A^+_k(K)}
\newcommand{\cAA}{\mathcal{A}^+_{i, p/q}(K)}
\newcommand{\cB}{B^+}
\newcommand{\cBB}{\mathcal{B}^+}
\newcommand{\cv}{v_k}
\newcommand{\ch}{h_k}
\newcommand{\hv}{\boldsymbol{v}_k}
\newcommand{\hvt}{\boldsymbol{v}_k^T}
\newcommand{\hvr}{\boldsymbol{v}_k^{red}}
\newcommand{\hh}{\boldsymbol{h}_k}
\newcommand{\hA}{\boldsymbol{A}^+_k(K)}
\newcommand{\hAo}{\boldsymbol{A}^+_0(K)}
\newcommand{\hB}{\boldsymbol{B}^+}
\newcommand{\T}{\mathcal{T}^+_d}
\newcommand{\To}{\mathcal{T}^+_0}
\newcommand{\Ta}{\mathcal{T}^+}
\newcommand{\hAA}{\mathbb{A}^+_{i,p/q}(K)}
\newcommand{\hBB}{\mathbb{B}^+}
\newcommand{\hAT}{\boldsymbol{A}^T_k(K)}
\newcommand{\hATo}{\boldsymbol{A}^T_0(K)}
\newcommand{\hAr}{\boldsymbol{A}^{red}_k(K)}
\newcommand{\hAro}{\boldsymbol{A}^{red}_0(K)}
\newcommand{\hAAT}{\mathbb{A}^T_{i,p/q}(K)}
\newcommand{\hAAr}{\mathbb{A}^{red}_{i,p/q}(K)}
\newcommand{\cD}{D^+_{i,p/q}}
\newcommand{\hD}{\boldsymbol{D}^+_{i,p/q}}
\newcommand{\hDT}{\boldsymbol{D}^T_{i,p/q}}
\newcommand{\hDr}{\boldsymbol{D}^{red}_{i,p/q}}
\newcommand{\cAn}{A^+_{\floor{\frac{i+pn}{q}}}(K)}

\newcommand{\ZZ}{\mathbb{Z}}
\newcommand{\MC}{\mathbb{X}^+_{i, p/q}}
\newcommand{\FF}{\mathbb{F}}
\newcommand{\FR}{\mathbb{F}[U]}
\newcommand{\KS}{S^3_{p/q}(K)}
\newcommand{\dU}{d(L(p,q),i)}
\newcommand{\dK}{d(\KS ,i)}
\newcommand{\hATn}{\boldsymbol{A}^T_{\floor{\frac{i+pn}{q}}}(K)}
\newcommand{\hArn}{\boldsymbol{A}^{red}_{\floor{\frac{i+pn}{q}}}(K)}

\newcommand{\HF}{HF^+(\KS ,i)}
\newcommand{\HFF}{HF^+(\KS )}

\newcommand{\HFr}{HF_{red}(\KS ,i)}
\newcommand{\HFFr}{HF_{red}(\KS )}
\newcommand{\As}{\mathbb{A}^s(K)}
\newcommand{\hAq}{\boldsymbol{A}^{red}_{\floor{\frac{i}{q}}}(K)}
\newcommand{\F}{F^+_{;i}}
\newcommand{\Fo}{F^+_{0;i}}
\newcommand{\Fm}{F^+_{m;i}}
\newcommand{\FO}{F^+_{;0}}
\newcommand{\FOo}{F^+_{0;0}}
\newcommand{\FOm}{F^+_{m;0}}
\newcommand{\Sp}{$\mathrm{Spin}^c$-structure}
\newcommand{\Sps}{$\mathrm{Spin}^c$-structures}

\newcommand{\cfk}{$CFK^{\infty}(K)$}

\def\co{\colon\thinspace}

\def\subrangle#1{\stackengine{5pt}{}{$\!\scriptstyle #1$}{U}{l}{F}{F}{L}}

\newtheorem{theorem}{Theorem}

\newtheorem{lemma}[theorem]{Lemma}

\newtheorem{prop}[theorem]{Proposition}

\newtheorem{defn}[theorem]{Definition}

\newtheorem{cor}[theorem]{Corollary}

\newtheorem{q}[theorem]{Question}

\begin{document}

\begin{abstract}
We write down an explicit formula for the $+$ version of the Heegaard Floer homology (as an absolutely graded vector space over an arbitrary field) of the results of Dehn surgery on a knot $K$ in $S^3$ in terms of homological data derived from \cfk. This allows us to prove some results about Dehn surgery on knots in $S^3$. In particular, we show that for a fixed manifold there are only finitely many alternating knots that can produce it by surgery. This is an improvement on a recent result by Lackenby and Purcell. We also derive a lower bound on the genus of knots depending on the manifold they give by surgery. Some new restrictions on Seifert fibred surgery are also presented.
\end{abstract}

\maketitle

\section{Introduction}
\label{sec:intro}

Dehn surgery is a fundamental technique in $3$-manifold topology. Indeed, we can construct any $3$-manifold\footnote{In this paper, whenever we say `$3$-manifold' we mean `closed connected orientable $3$-manifold'.} beginning with any other $3$-manifold and performing Dehn surgery enough times. However, it is a highly non-trivial and widely open problem to understand what manifolds can be obtained by doing Dehn surgery once (even starting from the `simplest' $3$-manifold, namely $S^3$) and what knots yield a fixed manifold by surgery.

Heegaard Floer theory is a relatively recent collection of powerful tools in low-dimensional topology. It has many aspects and provides invariants in many different contexts. In this paper, we are only concerned with the 3-manifold and knot invariants (defined in \cite{OSzTopInv,rasmussenThesis,OSzKnotInv}). The collections of 3-manifold invariants and knot invariants are connected via the surgery formula that expresses the Heegaard Floer homology of a 3-manifold obtained by surgery on a given knot in terms of the Heegaard Floer homology data of the knot (see \cite{OSzRatSurg}). This makes Heegaard Floer homology an especially suitable tool for investigating questions about Dehn surgery.

A natural question about Dehn surgery is whether there are manifolds that can be obtained by surgery on infinitely many distinct knots in $S^3$. The answer is `yes' -- see \cite{osoinach} or \cite{teragaito}. There is still hope, however, that perhaps this does not happen for some nice classes of knots.

One interesting and well-studied class of knots is that of alternating knots. At first sight, their diagrammatic definition seems to have little to do with the geometric-topological properties of these knots. However, this is not so -- see \cite{lackenbyPurcell} and references therein. In particular, the following is \cite[Theorem 1.3]{lackenbyPurcell}

\begin{theorem}[Lackenby-Purcell]
For any closed $3$-manifold $M$ with sufficiently large Gromov norm, there
are at most finitely many prime alternating knots $K$ and fractions $p/q$ such that $M$ is
obtained by $p/q$ surgery along $K$.
\label{thm:lackenby}
\end{theorem}

In fact, the statement about fractions $p/q$ can be deduced, for example, from \cite{niWu}. We will also show in this paper that given any manifold $Y$ there is a universal bound on $q$ for such fractions which also implies that they are finite in number. Using techniques that are very different from those used in \cite{lackenbyPurcell} we are able to establish the following improvement of this theorem.

\begin{theorem}
Let $Y$ be a $3$-manifold. There are at most finitely many alternating knots $K \subset S^3$ such that $Y = \KS$.
\label{thm:alternating}
\end{theorem}

Heegaard Floer homology is also very useful in bounding genera of various surfaces. In particular, knot Floer homology determines the genus of a knot \cite{OSzGenusBounds}. Combining this with information about surgery often allows one to put restrictions on genera of knots admitting certain surgeries. For example, if surgery on a knot $K$ produces an $L$-space $Y$ (a generalisation of lens spaces -- see below for the definition), then $2g(K)-1 \leq |H_1(Y)|$, where by $g(K)$ we mean the genus of $K$ (see e.g. \cite[Corollary 1.4]{OSzRatSurg}).

We derive a bound which is in some respects `opposite' to the bound for $L$-spaces. It is a lower bound which can be non-trivial only for non-$L$-spaces. For the statement of the theorem below and the rest of the paper note that we work over an arbitrary field $\FF$. Heegaard Floer homology is then an $\FR$-module and we denote the action of $U$ simply by multiplication. For a rational homology sphere $Y$, $HF_{red}(Y)$ denotes its reduced Floer homology.

\begin{theorem}
For any knot $K \subset S^3$ and any $p/q \in \mathbb{Q}$ we have
$$
U^{g(K) + \ceil{g_4(K)/2}}\cdot\HFFr = 0.
$$
\label{thm:genus}
\end{theorem}

We remark that if $K$ is an $L$-space knot, then $U^{\ceil{g_4(K)/2}}\cdot\HFFr = \nolinebreak 0$. Moreover, for any $N>0$ and $p>0$ there is a three-manifold $Y$ which can be obtained by a surgery on a knot in $S^3$ such that $U^N\cdot HF_{red}(Y) \neq 0$ and $|H_1(Y)| =\nolinebreak p$.

Here $g_4(K)$ is the slice genus of $K$. We obviously have $\ceil{\frac{g_4(K)}{2}} \leq \ceil{\frac{g(K)}{2}}$, so the theorem does give a lower bound for $g(K)$.

A different lower bound for the knot genus producing non-$L$-spaces has been found by Jabuka in \cite{jabukaGenus}, but unlike our bound, it also depends on the denominator of the slope. Note also that there exists a manifold for which the genus of knots producing it is not bounded above \cite{teragaito}.

More recently, Jabuka \cite{jabukaHat} has produced a new lower bound on the genus that does not involve the denominator of the slope. He also obtained the ranks of $\widehat{HF}$ for the result of surgery on a knot in $S^3$. His genus bound appears to be quite different from ours.

Using the genus bound of Theorem \ref{thm:genus} and some other considerations we are able to prove results about Seifert fibred surgery on knots in $S^3$. In \cite{wuCones} Wu (improving on the results of \cite{OSzSeifFibr}) has proven the following (the definitions of Seifert orientation and torsion coefficients will be provided later).

\begin{theorem}[Wu]
Let $K \subset S^3$ be a knot. Suppose there is a rational number $p/q > 0$ such that $Y = \KS$ is Seifert fibred. 

If $Y$ is a positively oriented Seifert fibred space, then all the torsion coefficients $t_i(K)$ are non-negative and $\widehat{HFK}(K,g(K))$ is supported in even degrees. In particular, $\deg \Delta_K = g(K)$.

If $Y$ is a negatively oriented Seifert fibred space and $0<p/q<3$, then for all \linebreak $i>0$ the torsion coefficients $t_i(K)$ are non-positive. If $Y$ is a negatively oriented Seifert fibred space, $g(K)>1$ and $2g(K)-1 > p/q$, then $\widehat{HFK}(K,g(K))$ is supported in odd degrees. In particular, $\deg \Delta_K = g(K)$.
\label{thm:wu}
\end{theorem}

We are able to prove the following.

\begin{theorem}
Let $K \subset S^3$ be a knot. Suppose there is a rational number $p/q > 0$ such that $Y = \KS$ is a negatively oriented Seifert fibred space. Then

\begin{itemize}
\item $U^{g(K)}\cdot HF_{red}(Y) = 0$;
\item if $0<p/q\leq 3$, then all the torsion coefficients $t_i(K)$ are non-positive (including $t_0(K)$) and $\deg \Delta_K = g(K)$;
\item more generally, if $i\geq \floor{\frac{\ceil{p/q}-\sqrt{\ceil{p/q}}}{2}}$, then $t_i$ is non-positive;
\item if $g(K) > \floor{\frac{\ceil{p/q}-\sqrt{\ceil{p/q}}}{2}}$, then $\deg \Delta_K = g(K)$;
\item if $U^{\floor{|H_1(Y)|/2}}\cdot HF_{red}(Y) \neq 0$ then $\deg \Delta_K = g(K)$.
\end{itemize}

In all statements where $\deg \Delta_K = g(K)$ we have that $\widehat{HFK}(K,g(K))$ is supported in odd degrees.
\label{thm:seifert}
\end{theorem}

After the proof of Theorem \ref{thm:genus} in Section \ref{sec:genus}, we describe negatively oriented Seifert fibred spaces $Y$ for which the power of $U$ needed to annihilate $HF_{red}(Y)$ gets arbitrarily large compared to the order of the first homology group.

Theorem \ref{thm:seifert} combined with the result of Wu has the following straightforward corollary.

\begin{cor}
Suppose $Y = \KS$ is a Seifert fibred rational homology sphere. If $|H_1(Y)| \leq 3$, then all the torsion coefficients of $K$ have the same sign and $\deg \Delta_K = g(K)$.
\label{cor:seifert_small_H1}
\end{cor}

To prove Theorems \ref{thm:alternating} and \ref{thm:genus} we need to study the mapping cone formula, which connects the Heegaard Floer data of the knot with the Heegaard Floer homology of the manifolds obtained by surgery on it. Given a knot $K$ in $S^3$ there is a doubly-filtered complex $C = $ \cfk \ associated to it. The doubly-filtered homotopy type of this complex is a knot invariant, from which all the flavours of knot Floer homology are derived.

In fact, the mapping cone formula states that given $C$ and a certain chain homotopy equivalence which identifies $C\{i\geq 0\}$ with $C\{j\geq 0\}$ we can determine $\HFF$ completely for any rational $p/q$.

In Section \ref{sec:calculations} we derive an explicit description of $\HFF$ as an absolutely graded vector space in terms of homological data from \cfk, with no reference to the chain homotopy equivalence mentioned above. For a large part this has already been done (\cite{niZhangChar}, \cite{niWu}, \cite{OSzRatSurg}), but the results are scattered across multiple papers, sometimes not in explicit form, and we consider it useful to have them collected in one place. While all the results of this section concerning positive surgeries have been shown before, as far as we are aware, the results for negative and zero surgeries (contained in subsections 3.2 and 3.3 respectively) are new.

This allows us to derive some other applications as well, a few of which we mention here.

\begin{theorem}
Suppose $K$ is a non-trivial knot and $Y = \KS$. Then 
$$
|q| \leq |H_1(Y)| + \dim(HF_{red}(Y)).
$$
\label{thm:q-bound}
\end{theorem}

The existence of such a bound for the denominator seems to be known to some experts in Heegaard Floer homology (it could be deduced from \cite[Proposition 9.6]{OSzRatSurg}) but we have not seen it explicitly stated. We will use this fact in the proof of Theorem \ref{thm:alternating}.

A bound on the number of slopes that produce a given manifold has also been obtained in \cite[Theorem 2.9]{lackenbyDehnSurg} in quite a general but somewhat different setting (in particular due to homological conditions it does not deal with surgeries in $S^3$).

\begin{theorem}
Let $K$ be an $L$-space knot and $p/q \leq 1$ a rational number. Then $\KS$ and $p/q$ determine the Alexander polynomial of $K$.
\label{Lsurg}
\end{theorem}

We remark that the Alexander polynomial determines the knot Floer homology of an $L$-space knot \cite{OSzLensSpaceSurg}. Conversely, the Alexander polynomial of an $L$-space knot determines the $HF^+$ of all surgeries on it. In particular, if for two $L$-space knots $K$ and $K'$ we have that $HF^+(S^3_{p/q}(K)) \cong HF^+(S^3_{p/q}(K'))$ for some $p/q \leq 1$, then $HF^+(S^3_{p'/q'}(K)) \cong HF^+(S^3_{p'/q'}(K'))$ for all other $p'/q'$.

This theorem also implies that if two torus knots $T_{r,s}$, $T_{r',s'}$ with $r,s,r',s'>0$ have the same surgery with the same slope $\leq 1$, then they are the same. Presumably, this can also be obtained by more elementary methods. Note, however, that there do exist positive integral slopes for which there are two distinct torus knots with the same surgeries at these slopes \cite[Example 1.1]{niZhangChar}.

In \cite{teragaito} Teragaito constructs a small Seifert fibred space $Y$ and a sequence of knots $K_n \subset S^3$ such that $K_n(-4) = Y$.\footnote{In fact, Teragaito constructs $-Y$, his knots are the mirror images of $K_n$ and the slope he uses is $4$. It is more convenient for us to work with this orientation.} Moreover, the genus of the knots $K_n$ is unbounded. Incidentally, this shows that we cannot hope for an upper bound on knot genus for knots giving some arbitrary manifold by surgery. In section \ref{sec:other} we show that $Y$ can only be obtained by $(-4)$-surgery and we find the Alexander polynomial of the knots $K_n$. It is, in fact, possible to find the Heegaard Floer homology of all manifolds obtained by surgery on each $K_n$.

The organisation of this paper is as follows. In section \ref{sec:mapping_cone} we review some definitions, notation and the mapping cone formula. In section \ref{sec:calculations} we derive the expression for the Heegaard Floer homology of surgeries on a knot. In section \ref{sec:alternating} we prove Theorem \ref{thm:alternating}, in section \ref{sec:genus} we prove Theorem \ref{thm:genus} and in section \ref{sec:seifert} we prove Theorem \ref{thm:seifert}. Finally, in section \ref{sec:other} we present some other applications of the mapping cone formula to questions in Dehn surgery.

\section*{Acknowledgements}

I would like to thank Tye Lidman, Jake Rasmussen, Andr\'as Juh\'asz, Marc Lackenby and Duncan McCoy for their very valuable suggestions and comments on the earlier drafts of this paper (Tye Lidman, in particular, suggested that Theorem \ref{thm:alternating} could be proven using techniques of this paper). I am also very grateful to Marco Marengon, Tom Hockenhull and V.S.Pyasetkii for many important comments on the structure and presentation of this paper. This paper greatly benefited from a visit to the University of Texas at Austin, and many interesting and enlightening conversations that I had there. For this opportunity, I am very thankful to the Doris Chen Award, and the help of my supervisor Dorothy Buck. I am particularly grateful to Dorothy for her continued encouragement and support over the course of my Ph.D. studies. Finally, I
would like to thank the referee for helpful comments on the drafts of this paper.

\section{The mapping cone formula}
\label{sec:mapping_cone}

In this section, we set up notation and review the rational surgery formula of Ozsv\'ath and Szab\'o \cite{OSzRatSurg}. We largely follow the exposition and notation of Ni and Wu in \cite{niWu}.

Given a knot $K$ in $S^3$ we can associate to it a doubly-filtered complex \linebreak $C = CFK^{\infty}(K)$. We denote generators of this complex by $[\boldsymbol{x}, i , j]$, where this generator has filtration $(i,j) \in \ZZ \times \ZZ$. By \cite[Lemma 4.5]{rasmussenThesis} the complex $C$ is homotopy equivalent (as a filtered complex) to a complex for which all filtration-preserving differentials are trivial. In other words, at each filtration level we replace the group, viewed as a chain complex with the filtration preserving differential, by its homology. From now on we work with this, \slshape reduced \upshape complex.

The complex $C$ is invariant under the shift by the vector $(-1,-1)$. There is an action of a formal variable $U$ on $C$ which is simply the translation by the vector $(-1,-1)$. In other words, the group at the filtration level $(i,j)$ is the same as the one at the filtration level $(i-1, j-1)$ and $U$ is the identity map from the first one to the second. Of course, $U$ is a chain map. In $C$ the map $U$ is invertible (but note that it will not be in various subcomplexes and quotients), so $C$ is an $\mathbb{F}[U, U^{-1}]$-module.

This means that as an $\mathbb{F}[U, U^{-1}]$-module $C$ is generated by the elements with the first filtration level $i = 0$. In the reduced complex the group at filtration level $(0, j)$ is denoted $\widehat{HFK}(K,j)$ and is known as the knot Floer homology of $K$ (at Alexander grading $j$).

The complex $C$ is absolutely $\ZZ$-graded. In fact, the complex $C$ is the complex used to compute the ($\infty$-flavour of the) Heegaard Floer homology of $S^3$, the knot provides an additional filtration for it. By grading the Heegaard Floer homology of $S^3$ we obtain the grading on $C$. The map $U$ decreases this grading by 2.

Using the filtration on $C$ we can define the following quotients of it.
$$
\cA = C\{i\geq 0 \mbox{ or } j \geq k\}, \mbox{\ \ } k \in \mathbb{Z}
$$
and
$$
\cB = C\{i \geq 0\} \cong CF^+(S^3).
$$

We also define two chain maps, $\cv, \ch \co \cA \to \cB$. The first one is just the projection (i.e. it sends to zero all generators with $i < 0$ and acts as the identity map for everything else). The second one is the composition of three maps: firstly we project to $C\{j \geq k\}$, then we multiply by $U^k$ (this shifts everything by the vector $(-k,-k)$) and finally we apply a chain homotopy equivalence that identifies $C\{j \geq 0\}$ with $C\{i \geq 0\}$. Such a chain homotopy equivalence exists because the two complexes both represent $CF^+(S^3)$ and by general theory \cite{OSzKnotInv} there is a chain homotopy equivalence between them, induced by the moves between the Heegaard diagrams. We usually do not know the explicit form of this chain homotopy equivalence.

Genus detection alluded to before has the following form \cite[Theorem 1.2]{OSzGenusBounds}.

\begin{theorem}[Ozsv\'ath-Szab\'o]
Let $K \subset S^3$ be a knot. 

Then $g(K) = \max\{j \in \ZZ | \widehat{HFK}(K,j) \neq 0\}$.
\label{thm:genus_detect}
\end{theorem}

From this (together with symmetries of $C$) we can see that the maps $\cv$ (respectively $\ch$) are isomorphisms if $k \geq g$ (respectively $k \leq -g$).

We define chain complexes
$$
\cAA = \bigoplus_{n \in \mathbb{Z}}(n,\cAn), \ \cBB = \bigoplus_{n \in \mathbb{Z}}(n,\cB).
$$

The first entry in the brackets here is simply a label used to distinguish different copies of the same group. 
There is a chain map $\cD$ from $\cAA$ to $\cBB$ defined by taking sums of all maps $\cv, \ch$ with appropriate domains and requiring that the map $\cv$ goes to the group with the same label $n$ and $\ch$ increases the label by 1. Explicitly, $\cD(\{(k,a_k)\}_{k \in \ZZ}) = \{(k,b_k)\}_{k \in \ZZ}$ where $b_k = v^+_{\floor{\frac{i+pk}{q}}}(a_k)+h^+_{\floor{\frac{i+p(k-1)}{q}}}(a_{k-1})$.

Each of $\cA$ and $\cB$ inherits a relative $\ZZ$-grading from the one on $C$. Let $\MC$ denote the mapping cone of $\cD$. We fix a relative $\ZZ$-grading on the whole of it by requiring that the maps $\cv, \ch$ (and so $\cD$) decrease it by 1. The following is proven in \cite{OSzRatSurg}.

\begin{theorem}[Ozsv\'ath-Szab\'o]
\label{thm:mapping cone}
There is a relatively graded isomorphism of $\FR$-modules
$$
H_*(\MC ) \cong HF^+( \KS , i ).
$$
\end{theorem}

The index $i$ in $HF^+( \KS , i )$ stands for the \Sp. The numbering of \Sps\ we refer to is defined in \cite{OSzRatSurg}, but we do not need precise details of how to obtain this numbering for our purposes.

We can also determine the absolute grading on the mapping cone. The group $\cBB$ is independent of the knot. Now if we insist that the absolute grading on the mapping cone for the unknot should coincide with the grading of $HF^+$ of the surgery on it (i.e. $\dU$), this fixes the grading on $\cBB$. We then use this grading to fix the grading on $\MC$ for arbitrary knots -- this grading then is the correct grading, i.e. it coincides with the one $HF^+$ should have.

The mental picture we have of the mapping cone theorem is illustrated in Figure \nolinebreak \ref{cone_complex}. We have two rows of groups. The bottom row is just the row of identical groups $\cB$. The upper row consists of the various `hook' groups $\cA$. Specifically, if the surgery slope is $p/q$, in the upper row we repeat each group $q$ times. We then have vertical arrows pointing down for the maps $\cv$, and the arrows for the maps $\ch$ are slanted. More precisely, they go $p$ groups to the right (if $p$ is negative, this means $-p$ to the left). This creates $|p|$ subcomplexes, connected by a zig-zag set of arrows. Each such zig-zag subcomplex corresponds to a \Sp\ on the manifold that is the result of the surgery. To obtain the Heegaard Floer homology of this manifold we need to take the mapping cone of this chain map.

\begin{figure}
\includegraphics[scale=0.8, clip = true, trim = 65 725 60 0]{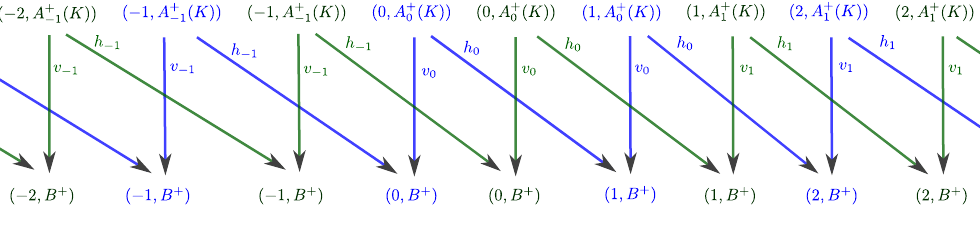}
\caption{Schematic representation of the portion of the complex, mapping cone of which gives the Heegaard Floer homology of the surgery on a knot. This case illustrates $\frac{2}{3}$-surgery, blue and green subcomplexes represent two different \Sps\ on the resulting space.}
\label{cone_complex}
\end{figure}

For our purposes, it suffices to pass to the homology of the mapping cone under consideration. Let $\hA = H_*(\cA ), \ \hB = H_*(\cB ), \ \hAA = H_*(\cAA), \ \hBB = H_*(\cBB)$ and let $\hv, \hh, \hD$ denote the maps induced by $\cv, \ch, \cD$ (respectively) in homology.

When we talk about $\hAA$ as an absolutely graded group, we mean the grading that it inherits from the absolute grading of the mapping cone that we described above.

Since $\cB \cong CF^+(S^3), \ \hB \cong \T$, where $\T \cong \mathbb{F}[U^{-1}] = \mathbb{F}[U, U^{-1}]/U\mathbb{F}[U]$, $d$ signifies the grading of 1 and multiplication by $U$ decreases the grading by 2. We sometimes call this module the \slshape tower. \upshape When we are not interested in the absolute grading we omit the subscript.

Recall that the short exact sequence
{\setlength\mathsurround{0pt}
\begin{center}
\begin{tikzcd}

0 \arrow{r}
&\cBB \arrow{r}{i}
&\MC \arrow{r}{j}
&\cAA \arrow{r}
&0
\end{tikzcd}
\end{center}
}
induces the exact triangle

{\setlength\mathsurround{0pt}
\begin{center}
\begin{equation}
\begin{tikzcd}
\hAA \arrow{r}{\hD}\arrow[leftarrow]{rd}{j_*}
&\hBB \arrow{d}{i_*}\\
&H_*(\MC ) \cong \HF.
\label{extriang}
\end{tikzcd}
\end{equation}
\end{center}
}

All maps in these sequences are $U$-equivariant. This triangle is the main tool in the calculations of the next section. In particular, if the surgery slope is positive, then the map $\hD$ will be surjective, so the triangle above implies that \linebreak $\HF \cong \ker(\hD)$.

\section{Calculations}
\label{sec:calculations}

In this section we want to use the mapping cone formula to calculate the Heegaard Floer homology for the results of surgery on a knot in $S^3$. Given a rational homology sphere $Y$ and $\mathfrak{s} \in$ \Sp, $\mbox{ } HF^+(Y, \mathfrak{s}) = \T \oplus HF_{red}(Y, \mathfrak{s})$, where $d = d(Y, \mathfrak{s})$ is called the \slshape correction term \upshape and $HF_{red}(Y, \mathfrak{s})$ is a finite-dimensional $\FR$-module annihilated by a big enough power of $U$, called the reduced Floer homology of $Y$ in \Sp\ $\mathfrak{s}$. The sum of these groups over all \Sps\ is called the reduced Floer homology of $Y$ and is denoted by $HF_{red}(Y)$.

We state a weaker version of \cite[Theorem 2.3]{OSzIntegerSurgeries}.

\begin{theorem}[Ozsv\'ath-Szab\'o]
\label{thm:large surgeries}
There is an integer $N$ such that for all $m \geq N$ and all $i \in \mathbb{Z}/m\mathbb{Z}$ there is an isomorphism of relatively graded $\FR$-modules
$$
\hA \cong HF^+(K_m, i),
$$
where $k \equiv i \mbox{ (mod $m$)}$ and $|k| \leq m/2$.
\end{theorem}

In particular, each $\hA$ is an $HF^+$ of a rational homology sphere in a certain \Sp, hence by the previous paragraph we can decompose it as \linebreak $\hA \cong \hAT \oplus \hAr$, where $\hAr$ is a finite-dimensional vector space in the kernel of some power of $U$ and $\hAT \cong \Ta$.

We will need to talk about the Euler characteristic of the groups $\hAr$, so we need to fix an absolute $\ZZ/2\ZZ$-grading for them. We do so by requiring that for the purposes of this grading each group $\hAT$ lies entirely in grading $0$ and then using the relative $\ZZ/2\ZZ$-grading (induced by the parity of the relative $\ZZ$-grading) on $\hA$.

A rational homology sphere $Y$ is called an \slshape $L$-space \upshape if $HF_{red}(Y, \mathfrak{s}) = 0$ for all \Sps\ $\mathfrak{s}$. A knot $K \subset S^3$ is called an \slshape $L$-space knot \upshape if some positive surgery on it is an $L$-space. In fact it is known that a $p/q$ surgery on an $L$-space knot is an $L$-space if and only if $p/q \geq 2g(K)-1$ ($g(K)$ is, as usual, the genus of $K$). In particular, all large surgeries on $L$-space knots are $L$-spaces, hence for any $L$-space knot $K$ we have $\hAr = 0$ for all $k$.

In the same way we can decompose the complexes of the exact triangle \eqref{extriang}:
$$
\hAAT = \bigoplus_{n \in \ZZ}(n,\hATn), \mbox{ } \hAAr = \bigoplus_{n \in \ZZ}(n,\hArn).
$$

We can also decompose the map $\hD = \hDT \oplus \hDr$, where the first map is the restriction of $\hD$ to $\hAAT$ and the second one is the restriction to $\hAAr$. Note that for $L$-space knots $\hD = \hDT$.

Now the restriction of $\hv$ (respectively $\hh$) to $\hAT$ is a multiplication by some power of $U$. (This is because at large gradings these maps are isomorphisms.) We denote this power by $V_k$ (respectively $H_k$). Following are some useful properties of $V_k$ and $H_k$, see \cite{niWu}:

\begin{itemize}
\item $V_k \geq V_{k+1}$ for any $k \in \ZZ$;
\item $H_k \leq H_{k+1}$ for any $k \in \ZZ$;
\item $V_k = H_{-k}$ for any $k \in \ZZ$;
\item $V_k \to + \infty$ as $k \to - \infty$;
\item $H_k \to + \infty$ as $k \to +\infty$;
\item $V_k = 0$ for $k\geq g(K)$;
\item $H_k = 0$ for $k \leq -g(K)$.
\label{vhprop}
\end{itemize}

In other words, $V_k$ form a non-increasing unbounded sequence of non-negative numbers, which become zero at $g(K)$ and $H_k = V_{-k}$.

We now separate into three different cases. Firstly, we cover the case of positive surgery slope. Secondly, we treat negative surgeries. The third case is the zero surgery.

\subsection{Positive surgeries}
\label{subsec:positive}

The next lemma is used to establish that $\hDT$ is surjective when $p/q > 0$. We state it in a slightly more general form because in this form it also applies to other manifolds.

\begin{lemma} 
Let $X = Y = \bigoplus_{i \in \ZZ}(i,\Ta), X' = \bigoplus_{i \neq 0}(i,\Ta)$ and maps

$$
\alpha_i \co (i,\Ta) \to (i,\Ta),
$$

$$
\beta_i \co (i,\Ta) \to (i+1,\Ta)
$$

be multiplications by $U^{a_i}$ and $U^{b_i}$ respectively. Suppose further that 
\begin{itemize}
\item there is a number $N$ s.t. $a_i = 0$ for $i\geq N$, $b_i = 0$ for $i\leq - N$ and
\item $a_i \to + \infty$ as $i \to -\infty$, $b_i \to + \infty$ as $i \to + \infty$.
\end{itemize}
Define $D$ to be the sum of the maps $\alpha_i$ and $\beta_i$. Then the restriction of $D$ to $X'$ is surjective.
\label{lemma:surj}
\end{lemma}

The setting here is very similar to the one described by Figure \ref{cone_complex}. Only we choose one of the zig-zag complexes and all the groups in both the top and the bottom row are the towers -- see Figure \ref{fig:positive_complex}.

\begin{proof}
This is essentially what Ni and Wu prove in \cite[Lemma 2.8]{niWu}. We will show that for any $n \geq 0$ and $i\leq 0$, $(i,U^{-n})$ is in the image of the restriction of $D$ to $X'$. The conclusion will then follow by symmetry and linearity.

We clearly have $(i,U^{-n}) = \beta_{i-1}(i-1,U^{-n-b_{i-1}})$. Define $\xi = \{(i,\xi_i)\}_{i \in \ZZ} \in X'$ recursively by 

$$
  \xi_s=\begin{cases}
    0 & \text{if $s \geq i$},\\
    U^{-n-b_{i-1}} & \text{if $s = i-1$},\\
    (-1)^{s-i+1}U^{a_{s+1}-b_s}\cdot \xi_{s+1} & \text{otherwise}.
  \end{cases}
$$

In a way, after we set that $\xi_s = 0$ for $s \geq i$, this is the only possible definition (up to the kernel of $D$). This is because the arrow `slanted to the right' has to be used to cancel the rightmost element in the lower row, hence we know what element in its co-domain we have to choose so that it indeed cancels. This tells us what the image of the `vertical' arrow is and hence what the next `slanted' arrow has to cancel etc.

Since we have $a_{s+1}-b_s \to + \infty$ as $s \to -\infty,$ $\xi$ only has a finite number of non-zero coordinates and hence is a well-defined element of $X'$. It is also easy to see that $D(\xi) = (i,U^{-n})$.
\end{proof}

Let $\tau_d(N)$ be a submodule of $\T$ generated by $\{ U^{-n} \}_{0 \leq n \leq N-1}$. As before, we omit the subscript in the absence of the absolute grading.

The setting of the next lemma is less general, indeed we use more information about the numbers $V_k$ and $H_k$.

\begin{figure}
\includegraphics[scale=0.72, clip = true, trim = 30 750 60 10]{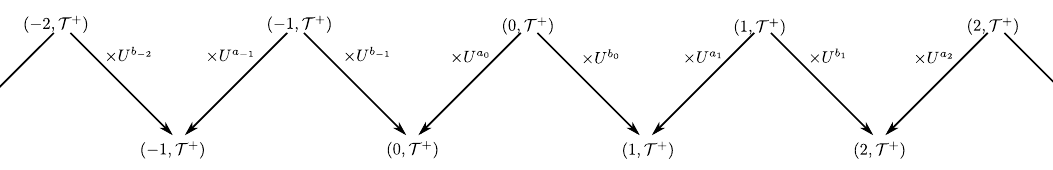}
\caption{Maps and groups of Lemmas \ref{lemma:surj} and \ref{lemma:surj_ker}}.
\label{fig:positive_complex}
\end{figure}

\begin{lemma}
To the assumptions of Lemma \ref{lemma:surj} add the following:
\begin{itemize}
\item $(a_i)$ is a non-increasing sequence;
\item $(b_i)$ is a non-decreasing sequence;
\item $a_i \leq b_i$ for $i \geq 0$;
\item $a_i \geq b_i$ for $i < 0$.
\end{itemize}
Put absolute gradings on $X$ and $Y$ by the rule that the maps $\alpha_i$ and $\beta_i$ decrease it by 1, the multiplication by $U$ decreases it by 2 and $1 \in (0,\Ta) \subset Y$ has grading $d-1$, where $d$ is some rational number.

Then if $a_0 \geq b_{-1}$ 
$$
\ker(D) \cong \Ta_{d-2a_0} \bigoplus_{n \geq 1} \tau_{d^-_n}(b_{-n}) \bigoplus_{n \geq 1} \tau_{d^+_n}(a_{n}).
$$
Otherwise
$$
\ker(D) \cong \Ta_{d-2b_{-1}} \bigoplus_{n \geq 2} \tau_{d^-_n}(b_{-n}) \bigoplus_{n \geq 0} \tau_{d^+_n}(a_{n}).
$$

The isomorphisms are as absolutely graded $\FR$-modules. The numbers $d^{\pm}_n$ are defined by $d^{\pm}_0 = d-2\max\{a_0,b_{-1}\}$, $d^-_{n+1} = d^-_n + 2(a_{-n}-b_{-(n+1)})$ and \linebreak $d^+_{n+1} = d^+_n + 2(b_{n}-a_{n+1}).$
\label{lemma:surj_ker}
\end{lemma}

\begin{proof}
The two cases are completely analogous, so we will assume $a_0 \geq b_{-1}$. First, following \cite[proof of Proposition 1.6]{niWu} we define $\rho^{T} \co \Ta_{d-2a_0} \to \ker(D)$ as follows. If we write $\rho^{T}(\eta) = \{(s,\xi_s)\}_{s \in \mathbb{Z}}$, we set $\xi_0 = \eta$ and determine the other components by
$$
  \xi_s=\begin{cases}
    -U^{b_{s-1}-a_s}\xi_{s-1} & \text{if $s > 0$},\\
    -U^{a_{s+1}-b_s}\xi_{s+1} & \text{if $s < 0$}.
  \end{cases}
$$

In effect, we want to simply send the tower to the tower in the 0-component of the upper group. But it is not in the kernel of $D$, so we need to correct for that. In fact, we also want the map to be an $\FR$-module homomorphism, which is the reason for considering the cases $a_0 \geq b_{-1}$ and $a_0 < b_{-1}$ separately.

Notice that we always multiply by a non-negative power of $U$: if $s > 0$, $b_{s-1} \geq a_{s-1} \geq a_s$; if $s = -1$, this is the assumption $a_0 \geq b_{-1}$; if $s < -1$, $a_{s+1} \geq b_{s+1} \geq b_s$. Thus the map is indeed an $\FR$-module homomorphism.

As before, $\xi_s = 0$ if $|s|$ is very big, so the map is well-defined. The map $\rho^{T}$ is one-to-one because its 0-component is (i.e. $\xi_0 = \eta$).  It is also graded correctly (i.e. the map $\rho^{T}$ sends homogeneous elements of absolute grading $d$ to homogeneous elements of grading $d$) because $(0,U^{-a_0})\in X$ is sent to $(0,1) \in Y$ by $\alpha_0$, which has grading $d-1$. Thus $(0,1) \in X$ has grading $d-2a_0$, since to descend from $(0,U^{-a_0})\in X$ to $(0,1) \in X$ we need to multiply by $U^{a_0}$ and multiplication by $U$ has grading $-2$.

We have identified the tower in the kernel. Now we need to deal with the rest of it. Below we prove, that the rest of the kernel consist of the kernels of the maps $\alpha_i + \beta_i$ for each $i$, except the one at which the tower is situated (i.e. $i=0$). It is easy to see, that the kernel of $\alpha_i + \beta_i$ is isomorphic to $\tau(\min(a_i,b_i))$.

If $\nu = \{(s,\nu_s)\}_{s\in \mathbb{Z}} \in \ker(D)$, by subtracting elements in the image of $\rho^{T}$ we may assume that $\nu \in X'$, i.e. $\nu_0 = 0$. Without loss of generality there exists $s<0$ s.t. $\nu_s \neq 0$. To finish the proof we need to show that $U^{b_s}\cdot \nu_s = 0$ (recall that in this range $b_s \leq a_s$). Suppose this is not so and $0 \neq U^{b_s}\cdot \nu_s$. Since $\nu$ is in the kernel, it has to be cancelled by something. It follows that we must have $\beta_s(\nu_s)+\alpha_{s+1}(\nu_{s+1})=0$. Thus $0 \neq U^{b_s}\cdot \nu_s = -U^{a_{s+1}}\nu_{s+1} \Rightarrow 0 \neq U^{b_{s+1}}\nu_{s+1}$, as $a_{s+1} \geq b_{s+1}$ if $s<-1$. By proceeding in this way it follows that $\nu_0 \neq 0$, i.e. $\nu \not \in X'$ -- a contradiction.
\end{proof}

The two lemmas above can be readily translated into results about surgery. The $d$-invariant formula \eqref{d-inv} from the corollary below is \cite[Proposition 1.6]{niWu}.

\begin{cor}
If $p/q>0$ the map $\hDT$ is surjective. It follows that so is $\hD$ and we conclude that $\HF \cong \ker(\hD)$.

If $\floor{\frac{i}{q}} \leq -\floor{\frac{i-p}{q}}$, then
$$
\ker(\hDT) \cong \T \bigoplus_{n \geq 1} \tau_{d^-_n}(H_{\floor{\frac{i-np}{q}}})\bigoplus_{n \geq 1} \tau_{d^+_n}(V_{\floor{\frac{i+np}{q}}}).
$$
Otherwise
$$
\ker(\hDT) \cong \T \bigoplus_{n \geq 2} \tau_{d^-_n}(H_{\floor{\frac{i-np}{q}}})\bigoplus_{n \geq 0} \tau_{d^+_n}(V_{\floor{\frac{i+np}{q}}}).
$$

Here 
\begin{equation}
d = \dK = \dU -2\max\{V_{\floor{\frac{i}{q}}},H_{\floor{\frac{i-p}{q}}}\},
\label{d-inv}
\end{equation}

and 
$$
d^-_n = d+2\sum_{k=0}^{n-1}(V_{\floor{\frac{i-kp}{q}}} - H_{\floor{\frac{i-(k+1)p}{q}}}),
$$
$$
d^+_n = d+2\sum_{k=0}^{n-1}(H_{\floor{\frac{i+kp}{q}}} - V_{\floor{\frac{i+(k+1)p}{q}}}).
$$
\label{cor:dtpos}
\end{cor}
\begin{proof}
This is a straightforward application of Theorem \ref{thm:mapping cone} and Lemmas \ref{lemma:surj} and \ref{lemma:surj_ker} after renumbering of the groups and maps -- objects numbered with $\floor{\frac{i+np}{q}}$ correspond to the ones numbered with $n$ in Lemmas \ref{lemma:surj} and \ref{lemma:surj_ker}.

To fix the grading, note that the grading of $\hBB$ does not depend on the knot, but only on the surgery slope. Thus to grade it we can take the unknot $U$. For the unknot we have $V_i = 0$ for $i \geq 0$ and $V_i = i$ for $i < 0$. Hence $0 = V_{\floor{\frac{i}{q}}} \geq H_{\floor{\frac{i-p}{q}}} = 0$, and by the the same argument as we used for an arbitrary knot, the grading of 1 in $(0, \boldsymbol{A}^+_{\floor{\frac{i}{q}}}(U))$ is the $d$-invariant of the surgery, which we know to be $\dU$ in this case. Since $V_{\floor{\frac{i}{q}}} = 0$, we find that the grading of 1 in $(0,\boldsymbol{B}^+)$ is $\dU - 1$. This allows us to fix the $d$-invariants for all other knots.

We can fix $d^{\pm}_n$ by the fact that the maps $\hv$ and $\hh$ reduce it by 1 and the multiplication by $U$ reduces it by 2.
\end{proof}

As we noted before, for $L$-space knots $\hD = \hDT$. Let $K$ be a knot and $\Delta_K(T) = a_0 + \sum_i a_i (T^i+T^{-i})$ be its symmetrised Alexander polynomial, with normalisation convention $\Delta_K(1) = 1$. Define its \slshape torsion coefficients \upshape $t_i(K)$ for $i\geq 0$ by
\begin{equation*}
t_i(K) = \sum_{j\geq 1} ja_{i+j}.
\end{equation*}

Clearly, if we know all the torsion coefficients, we know the Alexander polynomial. For $L$-space knots, $V_k = t_k$ for $k \geq 0$ (this follows, for example, from \cite{OSzRatSurg}), so Corollary \ref{cor:dtpos} determines the Heegaard Floer homology of positive surgeries on an $L$-space knot in terms of its Alexander polynomial.

The next proposition expresses the Heegaard Floer homology of positive surgeries for arbitrary knots in terms of data from $CFK^{\infty}$. This proposition is essentially \cite[Proposition 3.5]{niZhangChar}.

\begin{prop}
As absolutely graded vector spaces,
$$
\ker(\hD) \cong \ker(\hDT) \oplus \hAAr.
$$
Moreover, $\ker(\hDT)$ is actually a submodule of $\ker(\hD)$.
\label{prop:hfipos}
\end{prop}

\begin{proof}
This is a straightforward exercise in linear algebra. 

Given vector spaces $U, V, W$ and linear maps $\rho_U \co U \to W, \ \rho_V \co V \to W$, such that $\rho_U$ is surjective, $\ker(\rho_U \oplus \rho_V) \cong \ker(\rho_U) \oplus V$.

There exists a map $\rho^*_U \co W \to V$ such that $\rho_U \circ \rho^*_U = \mathrm{id}_W$. In the graded situation we can make $\rho^*_U$ send homogeneous elements to homogeneous elements. Then we can define $T \co \ker(\rho_U) \oplus V \to \ker(\rho_U \oplus \rho_V)$ by $T(x \oplus y) = (x - \rho^*_U \circ \rho_V(y)) \oplus y.$ Since in our case $\rho_U \oplus \rho_V$ is graded, $T$ is an isomorphism of graded vector spaces.
\end{proof}

Let
$$
\As = \bigoplus_{k \in \ZZ} \hAr.
$$

This is a finite-dimensional vector space, as each $\hAr$ is and $\hAr = 0$ for $|k| \geq g(K)$. We define $\delta(K) = \dim(\As)$. Note that $\delta(K) = 0 \Leftrightarrow K$ is an $L$-space knot. The following proposition (which generalises \cite[Proposition 5.3]{niWu}) is \cite[Corollary 3.6]{niZhangChar}.

\begin{prop}
Let $K \subset S^3$ be a knot and $p/q>0$. Then
\begin{equation}
\dim(\HFFr) = q\delta(K) + qV_0 + 2q\sum_{i=1}^{g-1}V_i - \sum_{i=0}^{p-1}\max(V_{\floor{\frac{i}{q}}},H_{\floor{\frac{i-p}{q}}}).
\end{equation}
\label{prop:rankpos}
\end{prop}

\begin{proof}
Since
$$
\dim(\HFFr) = \sum_{i=0}^{p-1}\dim(\HFr),
$$

combining Proposition \ref{prop:hfipos} and Corollary \ref{cor:dtpos} we see that
\begin{multline*}
\dim(\HFFr) = \\
= \sum_{i \in \ZZ}\dim(\hAq) + \sum_{i \geq 0}V_{\floor{\frac{i}{q}}} + \sum_{i \geq 1}H_{\floor{\frac{-i}{q}}} - \sum_{i=0}^{p-1}\max(V_{\floor{\frac{i}{q}}},H_{\floor{\frac{i-p}{q}}}) = \\
= q\sum_{k \in \ZZ}\dim(\hAr) + q\sum_{i=0}^{g-1}V_i + q\sum_{i=-(g-1)}^{-1}H_i- \sum_{i=0}^{p-1}\max(V_{\floor{\frac{i}{q}}},H_{\floor{\frac{i-p}{q}}})= \\
= q\delta(K) + qV_0 + 2q\sum_{i=1}^{g-1}V_i - \sum_{i=0}^{p-1}\max(V_{\floor{\frac{i}{q}}},H_{\floor{\frac{i-p}{q}}}).
\end{multline*}
\end{proof}

Now we are ready to prove Theorem \ref{thm:q-bound}.

{
\renewcommand{\thetheorem}{\ref{thm:q-bound}}
\begin{theorem}
Suppose $K$ is a non-trivial knot and $Y = \KS$. Then 
$$
|q| \leq |H_1(Y)| + \dim(HF_{red}(Y)).
$$
\end{theorem}
\addtocounter{theorem}{-1}
}
\begin{proof}
This is an easy consequence of Ni-Zhang's formula of Proposition \ref{prop:rankpos} (by taking the mirror image we may assume $p/q >0$). We have
\begin{multline*}
\dim(\HFFr) + \sum_{i=0}^{p-1}\max(V_{\floor{\frac{i}{q}}},H_{\floor{\frac{i-p}{q}}}) = \\
= q\delta(K) + qV_0 + 2q\sum_{i=1}^{g-1}V_i \geq q(\delta(K) + V_0).
\end{multline*}

Recall that $\delta(K) = \dim(\As)$, so it is non-negative and $\delta(K) = 0$ if and only if $K$ is an $L$-space knot, in which case $V_k = 0$ if and only if $k\geq g(K)$, so for non-trivial $L$-space knots $V_0 \neq 0$. If $V_0 = 0$ then all $V$'s (and $H$'s) are zero and as $\delta(K) \neq 0$ by the previous sentence, we clearly get $q \leq \dim(\HFFr)$.

So suppose $V_0 \neq 0$. Then
\begin{multline*}
\dim(\HFFr) + pV_0 \geq \\
\geq \dim(\HFFr) + \sum_{i=0}^{p-1}\max(V_{\floor{\frac{i}{q}}},H_{\floor{\frac{i-p}{q}}}) \geq q(\delta(K) + V_0).
\end{multline*}

Finally we have
\begin{multline*}
q \leq \frac{\dim(\HFFr) + pV_0}{\delta(K) + V_0} = \\
 \frac{\dim(\HFFr)}{\delta(K) + V_0} + \frac{pV_0}{\delta(K) + V_0} \leq \dim(\HFFr) + p.
\end{multline*}
\end{proof}

\subsection{Negative surgeries} 

In the case when $p/q <0$ the map $\hD$ is no longer surjective. However, we can show that the cokernel consists of exactly the tower part and the kernel is the reduced Floer homology $\HFr$. We start with a general lemma, which is similar to Lemmas \ref{lemma:surj} and \ref{lemma:surj_ker}. The main difference is in that the $\beta_i$ maps go to the groups labelled with a smaller index.

\begin{lemma} 
Let $X = Y = \bigoplus_{i \in \ZZ}(i,\Ta)$ and maps $\alpha_i \co (i,\Ta) \to (i,\Ta), \linebreak \beta_i \co (i,\Ta) \to (i-1,\Ta)$ be multiplications by $U^{a_i}$ and $U^{b_i}$ respectively. Suppose further that $a_i, b_i$ have the following properties.
\begin{itemize}
\item There is a number $N$ s.t. $a_i = 0$ for $i\geq N$, $b_i = 0$ for $i\leq - N$;
\item $a_i \to + \infty$ as $i \to -\infty$, $b_i \to + \infty$ as $i \to + \infty$;
\item $a_i \geq b_i$ for $i < 0$, $a_i \leq b_i$ for $i \geq 0$.
\end{itemize}
Then no element of $(-1,\Ta) \subset Y$ is in the image of $D$ and $(-1,\Ta) \subset Y$ generates the cokernel of $D$.
The kernel of $D$ has the following form.
$$
\ker(D) \cong \bigoplus_{i \in \ZZ}\tau(\min(a_i,b_i)).
$$
\label{lemma:non-surj}
\end{lemma}
\begin{proof}
As all of the maps $\alpha_i, \ \beta_i$ are surjective, it is easy to see that the cokernel of $D$ is generated by the (equivalence classes of) elements in any one of $(i,\Ta) \subset Y$. Suppose $\eta = \{(s,\eta_s)\}_{s\in \mathbb{Z}} = D(\xi)$ with $\eta_s = 0$ for $s \neq -1$. Let $\xi = \{(s,\xi_s)\}_{s\in \mathbb{Z}}$.

Without loss of generality (by symmetry) we may assume that $\alpha_{-1}(\xi_{-1}) \neq 0$. Since $a_{-1} \geq b_{-1}$ it follows that $\beta_{-1}(\xi_{-1}) \neq 0$. Since $\eta_{-2} = 0 = \beta_{-1}(\xi_{-1}) + \alpha_{-2}(\xi_{-2})$, we have $\alpha_{-2}(\xi_{-2}) \neq 0 \Rightarrow \xi_{-2} \neq 0$. Continuing in the same way we conclude that $\xi$ is not supported on a finite set and hence no such $\xi$ can exist.

Similarly to the proof of Lemma \ref{lemma:surj_ker}, we want to show that the kernel of $D$ separates into the kernels of maps $\alpha_i + \beta_i$. This will finish the proof.

Now let $\xi = \{(s,\xi_s)\}_{s\in \mathbb{Z}} \in \ker(D)$. As before, without the loss of generality we assume there is $n<0$ such that $\beta_n(\xi_n) \neq 0$. Then $\alpha_{n-1}(\xi_{n-1}) \neq 0$, so $\beta_{n-1}(\xi_{n-1}) \neq 0$. Proceeding inductively we again reach a contradiction to $\xi$ being finitely supported.
\end{proof}

The previous lemma describes the action of $\hDT$ when $p/q<0$. We make this explicit in the next lemma.

\begin{lemma}
Let $p < 0, \ q > 0$. Then 

$$
\mathrm{coker}(\hDT) \cong \T,
$$
where $d = \dU$.
$$
\ker(\hDT) \cong \bigoplus_{n \geq 1} \tau_{d^-_n}(H_{\floor{\frac{i-np}{q}}})\bigoplus_{n \geq 0} \tau_{d^+_n}(V_{\floor{\frac{i+np}{q}}}).
$$

Here $d^+_0 = d+1-2H_{\floor{\frac{i}{q}}}$, $d^-_n = d^+_0 +2\sum_{k=0}^{n-1}(V_{\floor{\frac{i-kp}{q}}} - H_{\floor{\frac{i-(k+1)p}{q}}})$,\linebreak  $d^+_n = d^+_0+2\sum_{k=0}^{n-1}(H_{\floor{\frac{i+kp}{q}}} - V_{\floor{\frac{i+(k+1)p}{q}}})$.
\label{lemma:dtneg}
\end{lemma}
\begin{proof}
This is a straightforward application of Lemma \ref{lemma:non-surj}. Objects that are labelled with $\floor{\frac{i+np}{q}}$ in the mapping cone correspond to the ones labelled with $-n$ in Lemma \ref{lemma:non-surj}. In particular take, $a_n = V_{\floor{\frac{i-np}{q}}}$, $b_n = H_{\floor{\frac{i-np}{q}}}$. The grading comes from the fact that this works in the same way for the unknot (the towers in the cokernel coincide for all knots). Just as in Corollary \ref{cor:dtpos} we get the values of $d^{\pm}_n$ by the fact that the maps $\hv$, $\hh$ have grading $-1$ and the multiplication by $U$ has grading \nolinebreak $-2$.
\end{proof}

Just as Corollary \ref{cor:dtpos} is sufficient for positive surgeries on $L$-space knots, so is Lemma \ref{lemma:dtneg} for negative surgeries on $L$-space knots. We observe that in this case the Alexander polynomial also determines the Heegaard Floer homology of the surgeries. Lemma \ref{lemma:dtneg} also implies that negative $p/q$ surgeries on $L$-space knots have the same $d$-invariants as the lens space $L(p,q)$, so do not depend on the particular $L$-space knot. The next proposition extends our analysis to arbitrary knots.

\begin{prop}
Let $p < 0, \ q > 0$. As absolutely graded $\FR$-modules
$$
\mathrm{coker}(\hD) \cong \T.
$$
As absolutely graded vector spaces
$$
\HFr \cong \ker(\hD) \cong \ker(\hDT) \oplus \mathcal{A},
$$
where $\hAAr \cong \mathcal{A} \oplus \tau_{\delta}(N_{i,p/q})$, $\delta = \dU + \nolinebreak 1$ and $N_{i,p/q}$ is characterised by 
$$
d = d(S^3_{p/q}(K),i) = d(L(p,q),i) + 2N_{i,p/q}.
$$

In fact, $N_{i,p/q} = \max\{\overline{V}_{\floor{\frac{i}{q}}}, \overline{H}_{\floor{\frac{i+p}{q}}}\}$, where $\overline{V}_k, \ \overline{H}_k$ are for the mirror image of $K$ the same as $V_k, \ H_k$ are for $K$.
\label{prop:hfineg}
\end{prop}
\begin{proof}
Recall that no element in $(-1, \hB)$ is in the image of the map $\hDT$. Since $\hAAr$ lies in the kernel of the multiplication by a big enough power of $U$, so is its image under $\hD$. Hence $\hD$ only 'chops off' a finite piece of the tower. More precisely, let $N$ be the largest integer such that $U^{-N+1} \in (-1, \hB)$ appears as a term of some element $\eta$ in the image of $\hD$.

We claim that then $U^{-N+k}$ is also in the image for all $k \geq 1$. This is easily seen by an inductive argument: 1 is in the image, as $1 = U^{N-1}\eta$; $U^{-1}$ is, because 1 is and $U^{N-2}\eta$ is. Proceeding in the same way we establish the claim.

Thus the cokernel of $\hD$ is generated by $U^{-N-k}$ for $k \geq 0$, none of which are in its image. Thus the map $i_*$ from the exact triangle \eqref{extriang} injects $\langle\{U^{-N-k}\}_{k \geq 0}\rangle\subrangle{\FF}$ into $\HF$. Since $U^{-N+1} \in (-1, \hB)$ is in the image of $\hD$, it is in the kernel of $i_*$ and we have $U\cdot i_*(U^{-N}) = 0$. Hence the image of $i_*$ is exactly the tower $\T$ with $d = \dK$. By Lemma \ref{lemma:dtneg}, $1 \in (-1, \hB)$ has grading $\dU$, so $\dK = \dU + 2N$.

By the First Isomorphism Theorem and exactness of \eqref{extriang} 
$$
\ker(\hD) = \mathrm{im}(j_*) \cong \HF / \ker(j_*) = \HF / \mathrm{im}(i_*).
$$

Since $\mathrm{im}(i_*)$ is the tower, we have
$$
\ker(\hD) \cong \HF / \mathrm{im}(i_*) \cong \HFr.
$$

The rest is just linear algebra again. We can split $\hAAr$ into the part that goes isomorphically to the base of the tower, which is not in the image of $\hDT$ (i.e. $(-1, \hB) \cap \mathrm{im}(\hD)$) and the part that goes into the image of $\hDT$. We then proceed as in the proof of Proposition \ref{prop:hfipos}.

The fact that $N_{i,p/q} = \max\{\overline{V}_{\floor{\frac{i}{q}}}, \overline{H}_{\floor{\frac{i-p}{q}}}\}$ follows from taking the mirror image of $K$ and comparing with the already obtained formula for the correction terms from Corollary \ref{cor:dtpos}. We have

\begin{multline*}
2N_{i,p/q} = d(S^3_{p/q}(K),i) - d(L(p,q),i) = \\
 = - d(S^3_{-p/q}(m(K)),i) + d(L(-p,q),i) = \max\{\overline{V}_{\floor{\frac{i}{q}}}, \overline{H}_{\floor{\frac{i+p}{q}}}\},
\end{multline*}

where $m(K)$ is the mirror image of $K$.
\end{proof}

We can also express the total rank of $\HFr$ as follows.

\begin{prop}
We have
$$
\dim(\HFFr) = q\delta(K) + qV_0 + 2q\sum_{i=1}^{g-1}V_i - \sum_{i=0}^{p-1}N_{i,p/q}.
$$
\label{prop:rankneg}
\end{prop}
\begin{proof}
The proof is virtually the same as for Proposition \ref{prop:rankpos}.
\end{proof}

\subsection{Zero surgeries}

We now treat the case of zero surgeries. For the case of $L$-space knots the formula for the Heegaard Floer homology of the zero surgery was derived in \cite[Theorem 7.2]{OSzAbsGr}. The main tool we use is \cite[Theorem 9.19]{OSzPropApp}:
\begin{theorem}[Ozsv\'ath-Szab\'o]
There is a $U$-equivariant exact triangle
\begin{center}
\begin{equation}
{\setlength\mathsurround{0pt}
\begin{tikzcd}
HF^+(S^3) \arrow{r}{\F}\arrow[leftarrow]{rd}[swap]{\Fm}
&\bigoplus\limits_{j \equiv i (mod \ m)} HF^+(S^3_0(K),j) \arrow{d}{\Fo}\\
&HF^+(S^3_m(K),i).
\label{extriang1}
\end{tikzcd}
}
\end{equation}
\end{center}
Moreover, the map $\Fm$ is equal to the one induced by the surgery cobordism.
\end{theorem}

Given $i$ we can make $m$ in \eqref{extriang1} so big that 
$$
\bigoplus\limits_{j \equiv i (mod \ m)} HF^+(S^3_0(K),j) = HF^+(S^3_0(K),i).
$$
From now on we assume that $m$ is at least that large.

The group $\hAo \cong \hATo \oplus \hAro$ is relatively $\ZZ$-graded. If we fix an absolute $\mathbb{Q}$-grading for any element of $\hAo$, the relative grading will fix the absolute grading for all the elements. In  particular, it will absolutely grade $\hAro$.

In the statement of the next proposition (but not necessarily in the proof), we use the grading of $\hAro$ induced by grading the tower $\hATo$ in such a way that the grading of 1 is $\frac{1}{2}-2V_0$.

\begin{prop}
Let $k \neq 0$. Then as $\ZZ/2\ZZ$-graded vector spaces
\begin{equation}
HF^+(S^3_0(K),k) \cong \tau(V_{|k|}) \oplus \hAr.
\end{equation}

As absolutely $\mathbb{Q}$-graded vector spaces
\begin{equation}
HF^+(S^3_0(K),0) \cong \mathcal{T}^+_{-\frac{1}{2}+2\overline{V}_0} \oplus \mathcal{T}^+_{\frac{1}{2}-2V_0}\oplus \mathcal{A}.
\end{equation}

Here $\mathcal{A} \oplus \tau_{1/2}(\overline{V_0}) \cong \hAro$ as absolutely graded vector spaces, where the absolute grading of $\hAro$ is as described above.
\label{zerosurg}
\end{prop}

\begin{proof}
The first part is immediate from \cite[proof of Theorem 7.2]{OSzAbsGr}. Note that $HF^+(S^3_m(K),k) \cong \mathcal{T} \oplus \hAr$ (recall that we are assuming that $m$ is large). In \cite[proof of Theorem 7.2]{OSzAbsGr} Ozsv\'ath and Szab\'o show that the restriction of $\Fm$ to the tower part is surjective and its kernel is $\mathbb{F}[U^{-1}]/U^{-V_{|k|}}$. So we are done by the same elementary linear algebra as in the proof of Proposition \ref{prop:hfipos}.

For the second part, note that we can assign absolute gradings as we are dealing with a torsion \Sp. As shown in \cite[Theorem 10.4]{OSzPropApp}, $HF^{\infty}(S^3_0(K),0)$ is a direct sum of two copies of $\ZZ[U, U^{-1}]$ that lie in different relative $\ZZ/2\ZZ$-gradings. This is equivalent to saying that the difference of the absolute gradings between the elements from the different summands is always odd. As in the case of rational homology spheres, the exact sequence

$$
\ldots \to HF^-(Y, \mathfrak{s}) \to HF^{\infty}(Y, \mathfrak{s}) \to HF^+(Y, \mathfrak{s}) \to \ldots
$$

establishes that

$$
HF^+(S^3_0(K),0) \cong \mathcal{T}_{d_1} \oplus \mathcal{T}_{d_2} \oplus \mathcal{A},
$$

where $\mathcal{A} = HF_{red}(S^3_0(K),0)$ is a finitely generated $\FR$-module in the kernel of some large enough power of $U$.

In fact, combining \cite[Proposition 4.12]{OSzAbsGr} with the $d$-invariant formula of Ni-Wu stated in Corollary \ref{cor:dtpos} we obtain $d_1 = -\frac{1}{2}+2\overline{V}_0$, $d_2 = \frac{1}{2}-2V_0$.

The last step in the proof is determining $\mathcal{A}$. The maps $\FO$ and $\FOo$ from the exact triangle \eqref{extriang1} have gradings $-\frac{1}{2}$ and $\frac{m-3}{4}$ respectively by \cite[Lemma 7.11]{OSzAbsGr}. The map $\FOm$ is not graded but is a sum of graded maps, and the set of grading shifts of these maps is $\{\frac{1-m(2k-1)^2}{4}\}_{k \in \ZZ}$.

Since $HF^+(S^3) \cong \To$ and the grading of the map $\FO$ is $-\frac{1}{2}$, $\mathcal{T}^+_{\frac{1}{2}-2V_0}$ is not in the image of $\FO$, hence the map $\FOo$ is an isomorphism between $\mathcal{T}^+_{\frac{1}{2}-2V_0}$ and the tower part of $HF^+(S^3_m(K),0)$, which is equal to $\mathcal{T}^+_{\frac{m-1}{4}-2V_0}$ by Proposition \ref{prop:hfipos}. Hence the restriction of the map $\FOm$ to the tower part of $HF^+(S^3_m(K),0)$ is zero. As in the proof of Proposition \ref{prop:hfineg}, the restriction of $\FOm$ to $HF_{red}(S^3_m(K),0)$ maps a subgroup of the form $\tau(N)$ isomorphically to the base of the tower $HF^+(S^3) \cong \To$. By the grading considerations again we see that $N = \overline{V}_0$.

Recall from Proposition \ref{prop:hfipos} that $HF^+(S^3_m(K),0) \cong \mathcal{T}^+_{\frac{m-1}{4}-2V_0} \oplus \hAro$ (the grading here is such that the relative grading is as it should be). Let the maximal grading of a non-trivial element in $\hAro$ be $\frac{m-1}{4}-2V_0 + C$.

Consider one homogeneous summand of $\FOm$ with grading $\frac{1-m(2k-1)^2}{4}$. It maps the element of $\hAro$ of maximal grading to an element with grading 
$$
\frac{m-1}{4}-2V_0 + C + \frac{1-m(2k-1)^2}{4} = \frac{m(1-(2k-1)^2)-8V_0+4C}{4}.
$$

If $k \neq 0, 1$ we have $1-(2k-1)^2 < 0$ and so by making $m$ sufficiently large we can make sure that $\frac{m(1-(2k-1)^2)-8V_0+4C}{4} < 0$ and as all non-trivial elements in the image have grading $\geq 0$ this means that all components with $k \neq 0, 1$ are zero.

Thus we can assume that the map $\FOm$ has grading $\frac{1-m}{4}$. As discussed above the map $\FOm$ maps a subgroup of $\hAro$ of the form $\tau(\overline{V}_0)$ isomorphically to such a subgroup at the lower end of the tower $HF^+(S^3) \cong \To$. Therefore 1 in $\tau(\overline{V}_0)$ must have grading $\frac{m-1}{4}$.

The rest of $\hAro$ will be in the kernel of $\FOm$ and thus in the image of $\mathcal{A}$ by $\FOo$. Now noting that the grading of the map $\FOo$ is $\frac{m-3}{4}$ finishes the proof.
\end{proof}

Torsion coefficients of the Alexander polynomial of a knot describe the Euler characteristics of the groups $\hAr$, which we can see for example by combining Theorems 10.14 and 10.17 of \cite{OSzPropApp} (though a more direct proof is also possible). This has also been shown in \cite[Lemma 3.2]{niZhangChar}.

\begin{lemma} For $k\geq 0$
\begin{equation}
t_k(K) = V_k + \chi (\hAr).
\end{equation}
\label{lemma:euler}
\end{lemma}

Recall that the absolute $\ZZ/2\ZZ$ grading used to calculate the Euler characteristics here is fixed by the requirement that the tower $\hAT$ lies entirely in grading $0$.

\section{Proof of Theorem \ref{thm:alternating}}
\label{sec:alternating}

In this section we prove

{
\renewcommand{\thetheorem}{\ref{thm:alternating}}
\begin{theorem}
Let $Y$ be a $3$-manifold. There are at most finitely many alternating knots $K \subset S^3$ such that $Y = \KS$.
\end{theorem}
\addtocounter{theorem}{-1}
}

The strategy of our proof is as follows. We first want to restrict the possible Alexander polynomials of knots that yield a given $3$-manifold $Y$ by surgery. We then want to show that out of this restricted set, only finitely many can be Alexander polynomials of alternating knots. This will finish the proof, due to the next Proposition, which can be found in \cite[Proposition 5.1]{mooreStarkston}. We provide the proof for the reader's convenience (and since it is nice and short).

\begin{prop}[Moore-Starkston]
There is only a finite number of alternating knots with a given Alexander polynomial.
\label{prop:fin-Alex}
\end{prop}
\begin{proof}
By the Bankwitz Theorem \cite[Theorem 5.5]{crowell} the determinant $\det(K)$ of an alternating knot $K$ is greater than or equal to the minimal crossing number of $K$. Thus there are only finitely many alternating knots with a given determinant. The classical result \cite[page 213]{rolfsen} (or definition) $\det(K) = |\Delta_K(-1)|$ finishes the proof.
\end{proof}

For a knot $K \subset S^3$, let $\mathrm{m}(K)$ be its mirror image. Clearly, $K$ is alternating if and only if $\mathrm{m}(K)$ is. Since $S^3_{p/q}(K) = -S^3_{-p/q}(\mathrm{m}(K))$ we can assume that the surgery slope is positive (if non-zero).

For $Y$ a rational homology sphere and $q>0$ a natural number define 
$$
M(Y,q) = \frac{1}{2}(\sum_{0\leq i\leq p-1}d(L(p,q),i) - \sum_{\mathfrak{s}\in Spin^c(Y)}d(Y,\mathfrak{s})),
$$
where $p = |H_1(Y)|$.

Theorem \ref{thm:q-bound} shows that for any rational homology sphere $Y$ there is some number $n(Y)$ such that $Y \neq \KS$ for any $K$ and $|q| > n(Y)$.

If $Y$ is obtained by $p/q>0$ surgery on $K$, then by \eqref{d-inv} the numbers $V_k$ for $K$ satisfy
$$
M(Y,q) = \sum_{i=0}^{p-1}\max\{V_{\floor{\frac{i}{q}}},V_{-\floor{\frac{i-p}{q}}}\}.
$$

Combining this with Proposition \ref{prop:rankpos} we get
$$
\dim(HF_{red}(\KS)) + M(\KS,q) = q(\delta(K)+V_0+2\sum_{i\geq 1}V_i).
$$

This formula implies the following inequality:

$$
\frac{\dim(HF_{red}(\KS)) + M(\KS,q)}{q}\geq \sum_{k\geq 0}(V_k + \dim(\hA)).
$$

Now let
$$
c(Y) = \max_{1\leq q \leq n(Y)}\{\frac{\dim(HF_{red}(Y))+M(Y,q)}{q}\}.
$$

The inequality above implies, that if a rational homology sphere $Y$ is obtained by surgery on a knot $K$ with associated sequence $\{V_k\}_{k\geq 0}$, then
\begin{equation}
c(Y)\geq \sum_{k\geq 0}(V_k + \dim(\hA)).
\label{v-a-y}
\end{equation}

\begin{lemma}
Suppose $Y$ is a rational homology sphere obtained by a $p/q>0$ surgery on a knot $K \subset S^3$. Then
\begin{equation}
\sum_{i\geq 0}|t_i(K)|\leq c(Y).
\label{Y-bd}
\end{equation}
\label{torsBound}
\end{lemma}
\begin{proof}
It follows from Lemma \ref{lemma:euler} that for each $k\geq 0$
$$
|t_k(K)|=|V_k+\chi(\hA)|\leq V_k+|\chi(\hA)|\leq V_k + \dim(\hA).
$$
Combining with equation \eqref{v-a-y} yields the result.
\end{proof}

Let $S_Y$ be some set of knots in $S^3$ that give a rational homology sphere $Y$ by surgery (not necessarily all such knots and not necessarily alternating). Denote by $g(S_Y)$ ($\Delta(S_Y)$ respectively) the set of genera (Alexander polynomials respectively) of knots in $S_Y$.

\begin{lemma}
If $g(S_Y)$ is finite, then so is $\Delta(S_Y)$.
\label{g-Alex}
\end{lemma}

\begin{proof}
We clearly have $t_i(K) = 0$ for all $K \in S_Y$ and all $i \geq \max(g(S_Y))$. By Lemma \ref{torsBound}, $\sum_{i\geq 0}|t_i(K)|$ is bounded above, so we clearly have finitely many sequences $\{t_i(K)\}$ for $K \in S_Y$. Now observe that the torsion coefficients determine the Alexander polynomial so this results in at most finitely many possible Alexander polynomials.
\end{proof}

A theorem of Murasugi \cite[Theorem 1.1]{murasugiAlexAlt} is crucial for our proof:
\begin{theorem}[Murasugi]
Let $K\subset S^3$ be an alternating knot and 
$$
\Delta_K(T) = a_0 + \sum_{i=1}^{g(K)} a_i(T^i+T^{-i})
$$
be its Alexander polynomial. Then $a_i \neq 0$ for $0\leq i \leq g(K)$.
\label{murasugi}
\end{theorem}

The next Lemma is the last step before we can prove Theorem \ref{thm:alternating}.

\begin{lemma}
Let $K \subset S^3$ be an alternating knot that gives a rational homology sphere $Y$ by surgery. Then
$$
g(K) \leq 3c(Y).
$$
\label{final}
\end{lemma}
\begin{proof}
Suppose $g(K)\geq3c(Y)+1$. Note that $a_g = t_{g-1}(K) \neq 0$. We claim that there are three consecutive indices $i$, $i+1$ and $i+2\leq g$ with $t_i(K) = t_{i+1}(K) = t_{i+2}(K) = 0$. It then follows that $a_{i+1} = 0$, which is a contradiction to Theorem \ref{murasugi}.

To prove the claim suppose there is no such consecutive triple of zero torsion coefficients. Then
$$
\sum_{i\geq 0}|t_i(K)| = \sum_{k\geq 0}(|t_{3k}(K)|+|t_{3k+1}(K)|+|t_{3k+2}(K)|) \geq \floor{\frac{g-1}{3}}+1\geq c(Y)+1,
$$
which contradicts Lemma \ref{torsBound}.

We have thus established that $g\leq 3c(Y)$.
\end{proof}

\begin{proof}[Proof of Theorem \ref{thm:alternating}]
Suppose $Y$ is a rational homology sphere. Then by Lemma \ref{final} there is a genus bound for alternating knots that give $Y$ by surgery, so by Lemma \ref{g-Alex} the set of Alexander polynomials of such alternating knots is finite.

If $Y$ is obtained by $0$-surgery on $K$, then Propositions 10.14 and 10.17 of \cite{OSzPropApp} show that  the Alexander polynomial of $K$ can be deduced directly from the Heegaard Floer homology of $Y$.

Proposition \ref{prop:fin-Alex} now finishes the proof.
\end{proof}

\section{The genus bound}
\label{sec:genus}

We now turn to the proof of Theorem \ref{thm:genus}, which we restate here.

{
\renewcommand{\thetheorem}{\ref{thm:genus}}
\begin{theorem}
For any knot $K \subset S^3$ and any $p/q \in \mathbb{Q}$ we have
$$
U^{g(K) + \ceil{g_4(K)/2}}\cdot\HFFr = 0.
$$
\end{theorem}
\addtocounter{theorem}{-1}
}

\begin{lemma}
Let $K$ be a knot in $S^3$ with genus $g$. Then for any $k \in \ZZ$ 
$$
U^g\cdot \hAr = 0.
$$
\label{lemma:As-genus}
\end{lemma}
\begin{proof}
By the conjugation symmetry we may assume that $k \geq 0$. Let $C = CFK^{\infty}(K)$, $\Delta_k = C\{i<0 \text{ and } j\geq k\}$. This is a subquotient of $C$ (i.e. a subcomplex of a quotient). Note that $U^g\cdot \Delta_k = 0$, as this is the maximal possible 'height' of this complex. We illustrate the complexes $\Delta_k, \cA, \cB$ in Figure \ref{delta}.

We have an exact sequence
$$
0 \to \Delta_k \to \cA \to \cB \to 0
$$
which leads to an exact $U$-equivariant triangle

\begin{center}
\begin{equation}
{\setlength\mathsurround{0pt}
\begin{tikzcd}
H_*(\Delta_k) \arrow{r}{i_*}\arrow[leftarrow]{rd}
&\hA \arrow{d}{\hv}\\
&\hB.
\label{extriang2}
\end{tikzcd}
}
\end{equation}
\end{center}

Since $\hv$ is surjective we in fact have a short exact sequence
$$
0 \to H_*(\Delta_k) \to \hA \to \hB \to 0,
$$
so $H_*(\Delta_k) \cong \ker(\hv)$ and hence $U^g\cdot \ker(\hv) = 0$.

Recall that $\hA = \hAT \oplus \hAr$ and similarly we can decompose the map $\hv = \hvt \oplus \hvr$ into components. The map $\hvt$ is surjective. We claim that $\hAr \cong \ker(\hv)/\ker(\hvt)$. From this the conclusion of the Lemma follows immediately.

To prove the claim we construct an isomorphism from $\ker(\hv)/\ker(\hvt)$ to \linebreak $\hAr$. Let $x \in \ker(\hv) \setminus \ker(\hvt)$. Then send an equivalence class of $x$ to $\hAr$ by projection. This map is well-defined, because two different elements with the same projection are in $\ker(\hvt)$. Clearly this is also a surjective $\FR$-module homomorphism.
\end{proof}

\begin{figure}[h]
\includegraphics[scale=0.6, clip = true, trim = 30 320 60 0]{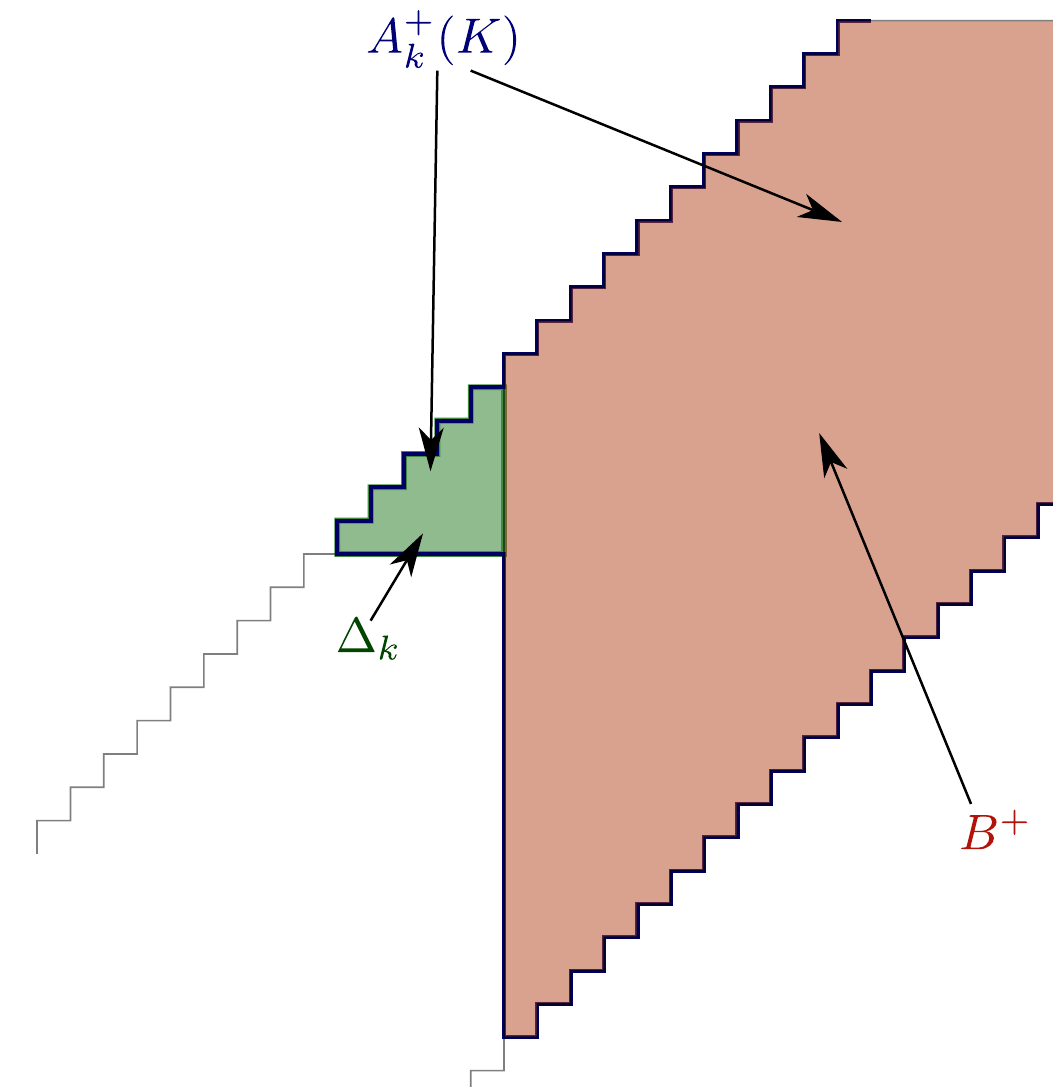}
\caption{Complexes $\Delta_k$, $\cA$ and $\cB$ inside $CFK^{\infty}$}
\label{delta}
\end{figure}

The previous lemma clearly implies

\begin{cor}
The following relation holds:
$$
U^g\cdot \hAAr = 0.
$$
\label{cor:reduced-stable}
\end{cor}

\begin{proof}[Proof of Theorem \ref{thm:genus}]
Note that by \cite[Theorem 2.3]{rasmussenGodaTerag} we have $V_0 \leq \ceil{\frac{g_4(K)}{2}}$.

If the slope is negative the reduced part is exactly equal to the kernel of $\hD$. So suppose $x\in \ker(\hD)$. By Corollary \ref{cor:reduced-stable}, $U^{g}\cdot x \in \ker(\hDT)$. But by Lemma \ref{lemma:dtneg} the kernel of $\hDT$ consist of the summands of the type $\tau(N)$ with $N \leq V_0 \leq \ceil{\frac{g_4(K)}{2}}$, so $U^{\ceil{\frac{g_4(K)}{2}}}\cdot \ker(\hDT) = 0$.

Now suppose the slope is positive. If we assume that $x \in \ker(\hD)$ then we still have $U^{g}\cdot x \in \ker(\hDT)$ and $U^{\ceil{\frac{g_4(K)}{2}}}\cdot \ker(\hDT)$ is contained in the tower part by Corollary \ref{cor:dtpos}.

Similarly, the case of zero surgery follows immediately from Proposition \ref{zerosurg}. This finishes the proof.
\end{proof}

Since by Corollary \ref{cor:dtpos} and Lemma \ref{lemma:dtneg} the reduced Floer homology of surgeries on $L$-space knots consists only of a direct sum of $\FR$-modules of the form $\tau(V_k)$, we see that if $K$ is an $L$-space knot, then $U^{\ceil{g_4(K)/2}}\cdot\HFFr = 0$.

In order to construct examples for which this genus bound gets arbitrarily large, note that every negative surgery on a knot contains a summand of the form $\tau(V_0)$. So if $V_0$ is large, the genus bound will also be large, independent of the absolute value of the negative slope we use. In particular, we can choose any order of the first homology we like.

For $L$-space knots, $V_0 = t_0$ can be read from the Alexander polynomial, in particular this is true for torus knots $T_{p,q}$ with $p,q >0$.

Suppose we have an $L$-space knot $K$ with Alexander polynomial 
$$
\Delta_K(T) = a_0 + \sum_{i=1}^ga_i(T^i+T^{-i}).
$$

Then the coefficients alternate between $1$ and $-1$, with the first non-trivial coefficient being $1$ \cite[Theorem 1.2]{OSzLensSpaceSurg}. So we clearly have 

$$
t_0 \geq \#\{a_i = 1, i>0\} \geq \frac{\#\{a_i \neq 0\}-1}{4}\geq\frac{\Delta_K(-1)-1}{4}.
$$

Consider the torus knots $T_{p,2}$ for $p$ positive odd. They have Alexander polynomials of the form
$$
\frac{(T^{2p}-1)(T-1)}{(T^p-1)(T^2-1)} = T^{p-1}-T^{p-2}+\ldots + 1,
$$
which evaluates to $p$ at $-1$.

Moreover, these examples are actually negatively oriented (see next section) small Seifert Fibred spaces, which is interesting in light of the next section.

We note that a result similar to Theorem \ref{thm:genus} can be obtained for a knot in any L-space rational homology sphere, the bound being in terms of the width of the knot Floer homology rather than the genus.

\section{Seifert fibred surgery}
\label{sec:seifert}

The aim of this section is to prove

{
\renewcommand{\thetheorem}{\ref{thm:seifert}}
\begin{theorem}
Let $K \subset S^3$ be a knot. Suppose there is a rational number $p/q > 0$ such that $Y = \KS$ is a negatively oriented Seifert fibred space. Then

\begin{itemize}
\item $U^{g(K)}\cdot HF_{red}(Y) = 0$;
\item if $0<p/q\leq 3$, then all the torsion coefficients $t_i(K)$ are non-positive (including $t_0(K)$) and $\deg \Delta_K = g(K)$;
\item more generally, if $i\geq \floor{\frac{\ceil{p/q}-\sqrt{\ceil{p/q}}}{2}}$, then $t_i$ is non-positive;
\item if $g(K) > \floor{\frac{\ceil{p/q}-\sqrt{\ceil{p/q}}}{2}}$, then $\deg \Delta_K = g(K)$;
\item if $U^{\floor{|H_1(Y)|/2}}\cdot HF_{red}(Y) \neq 0$ then $\deg \Delta_K = g(K)$.
\end{itemize}

In all statements where $\deg \Delta_K = g(K)$ we have that $\widehat{HFK}(K,g(K))$ is supported in odd degrees.
\end{theorem}
\addtocounter{theorem}{-1}
}
\begin{proof}
First we need to define the Seifert orientation for Seifert fibred spaces. Following \cite{OSzSeifFibr} we say that $Y$ has \slshape positive Seifert orientation \upshape if $-Y$ bounds $W(\Gamma)$, where $\Gamma$ is a weighted tree which has either negative-definite or negative-semi-definite intersection form. For the construction of the 4-manifold $W(\Gamma)$ from the weighted tree $\Gamma$ see \cite{OSzPlumbed}. We say that $Y$ has \slshape negative Seifert orientation \upshape if $-Y$ has positive Seifert orientation.

Using \cite[Corollary 1.4]{OSzPlumbed} (together with the inversion of the absolute $\ZZ/2\ZZ$-grading on the reduced homology upon reversing the orientation) we can see that if $Y$ has a negative Seifert orientation, then its reduced Floer homology is concentrated in the odd $\ZZ/2\ZZ$-grading and that it bounds a negative-definite 4-manifold with torsion free first homology group.

\begin{lemma}
Let $K \subset S^3$ be a knot. Suppose there is a rational number $p/q > 0$ such that $Y = \KS$ is a negatively oriented Seifert fibred space. Then $\hAr$ is supported in odd $\ZZ/2\ZZ$ grading for every $k$.
\label{lemma:this}
\end{lemma}

\begin{proof}[Proof of Lemma \ref{lemma:this}]
As an absolutely $\ZZ/2\ZZ$-graded group, each $\hAr$ is a subgroup of $\HFFr$ by Proposition \ref{prop:hfipos}. Since $\HFFr$ is supported in odd grading, so must each $\hAr$.
\end{proof}

Denote by $\widetilde{g}$ the minimal index $i$ for which $V_i = 0$. As previously, denote by $a_i$ the coefficient of the Alexander polynomial of $K$ corresponding to $T^i$. If $\widetilde{g}<g(K)$, then by Lemma \ref{lemma:euler}
\begin{equation}
a_{g(K)} = t_{g(K)-1} = \chi(\boldsymbol{A}^{red}_{g(K)-1}),
\label{ag}
\end{equation}

so, in particular, if all $\hAr$ are supported in the same $\ZZ/2\ZZ$-grading, then $a_{g(K)} \neq 0$, since  in this case 

$$
\boldsymbol{A}^{red}_{g(K)-1} \cong HF^+(S^3_0(K),g-1) \cong \widehat{HFK}(K,g(K)) \neq 0.
$$

It follows that in this case $\deg(\Delta_K) = g(K)$ and $\widehat{HFK}(K,g(K))$ is supported in odd degrees.

Moreover, if $\widetilde{g} = 0$, then $V_k = 0$ for all $k \geq 0$, so that
$$
t_k = \chi(\hAr) \leq 0.
$$

We now need to establish conditions which ensure that $\widetilde{g} = 0$ or $\widetilde{g}<g(K)$.

In \cite[Lemma 2.3]{mccoy} McCoy slightly modifies the proof of \cite[Theorem 1.1]{greeneCabling} by Greene to show that if $\KS$ bounds a negative-definite 4-manifold with torsion-free first homology, then
$$
2\widetilde{g}\leq n - \sqrt{n},
$$
where $n = \ceil{\frac{p}{q}}$.

It follows, that if $p/q \leq 3$, then $\widetilde{g} = 0$.

More generally, if $i \geq \floor{\frac{n-\sqrt{n}}{2}}$, where $n = \ceil{\frac{p}{q}}$, then $i \geq \widetilde{g}$ and hence $V_i = 0$. It follows that $t_i = \chi(\boldsymbol{A}^{red}_i(K))\leq 0$. If $g(K) > \floor{\frac{\ceil{p/q}-\sqrt{\ceil{p/q}}}{2}}$, then $g(K)>\widetilde{g}$ as well.

For the improvement of the genus bound, notice that all the summands of $\HFFr$ coming from $V_i$'s (i.e. of the form $\tau(V_i)$) are situated in the even grading and therefore must vanish. It now follows from the proof of Theorem \ref{thm:genus} that $U^{g(K)}\cdot\HFFr = 0$.

Now if $U^{\floor{|H_1(Y)|/2}}\cdot \HFFr \neq 0$, then $\floor{|H_1(Y)|/2}\leq g(K)-1$, so $\frac{|H_1(Y)|+1}{2}\leq g(K)$. On the other hand,
$$
\widetilde{g} \leq \frac{\ceil{p/q}-\sqrt{\ceil{p/q}}}{2} < \frac{p/q+1}{2}\leq \frac{p+1}{2} = \frac{|H_1(Y)|+1}{2}\leq g(K).
$$

It follows from \eqref{ag}, that $\deg(\Delta_K) = g(K)$.
\end{proof}

We end this section with the following

\begin{q}
Does there exist a knot $K \subset S^3$ with $\deg(\Delta_K) \neq g(K)$ and with a Seifert fibred surgery?
\end{q}

\section{Some other applications of the mapping cone formula}
\label{sec:other}

In this section, we demonstrate some other applications of the results obtained in Section \ref{sec:calculations}.

{
\renewcommand{\thetheorem}{\ref{Lsurg}}
\begin{theorem}
Let $K$ be an $L$-space knot and $p/q \leq 1$ a rational number. Then $\KS$ and $p/q$ determine the Alexander polynomial of $K$.
\end{theorem}
\addtocounter{theorem}{-1}
}

\begin{proof}
If the slope is zero this is immediate from Proposition \ref{zerosurg}. If the slope is negative this also easily follows from Lemma \ref{lemma:dtneg} -- by looking at $\HFFr$ we can work out a sequence of numbers that represents all the torsion coefficients with some repetitions (they are orders of cyclic $\FR$-modules). But we know the number of repetitions because we know the slope. From this we deduce all the torsion coefficients (in the correct order, as they form a monotone sequence), and hence the Alexander polynomial.

If the slope is in the interval $(0,1]$ the reasoning is the same -- Corollary \ref{cor:dtpos} allows us to work out the torsion coefficients, since we know how many times each occurs. The only torsion coefficient we might not be able to work out from the module structure of $\HFFr$ is $t_0$ if the slope is 1. But in this case, it can be worked out from the $d$-invariant formula of Ni-Wu from Corollary \ref{cor:dtpos}.
\end{proof}

Sometimes we can work out a lot about the Heegaard Floer homology associated to a knot from a surgery on it even if it is not an L-space knot.

\begin{prop}
The small Seifert Fibred space $Y = S^2((2,1),(6,-1),(7,-2))$ can only be obtained by $(-4)$-surgery. All knots producing it are non-$L$-space knots. 
\label{prop:terag}
\end{prop}
\begin{proof}
We find the $HF^+$ of this space using the computer program HFNem2 by \c{C}a\u{g}r{\i} Karakurt\footnote{At the time of writing available for download at \url{https://www.ma.utexas.edu/users/karakurt/}}. There are four \Sps $\{\mathfrak{s}_i\}_{i=0}^3$ and $HF^+$ in them have the form 
\begin{itemize}
\item[] $HF^+(Y,\mathfrak{s}_0) \cong \mathcal{T}_{-3/4}$
\item[] $HF^+(Y,\mathfrak{s}_1) \cong \mathcal{T}_0 \oplus \tau_0(1)$
\item[] $HF^+(Y,\mathfrak{s}_2) \cong \mathcal{T}_{1/4}$
\item[] $HF^+(Y,\mathfrak{s}_3) \cong \mathcal{T}_0 \oplus \tau_0(1)$
\end{itemize}

Using Theorem \ref{thm:q-bound} we can restrict the possible slopes to $\{\pm 4, \pm 4/3, \pm 4/5\}$. Calculating the correction terms of $L(4, 1) = L(4,-3) = L(4,5)$ and $L(4, -1) = L(4,3) = L(4,-5)$
we notice that only $L(4, -1)$ has correction terms such that the difference of each of them with some correction term of $Y$ is an integer. This means that the slope has to be in $\{-4,4/3,-4/5\}$.

We also notice that the $d$-invariants of $Y$ coincide exactly with the $d$-invariants of the lens space $L(4, -1)$. By the $d$-invariant formula \eqref{d-inv} we conclude that $V_0 = 0$. A similar argument using the $d$-invariant formula for negative surgeries in Proposition \ref{prop:hfineg} establishes that $\overline{V_0} = 0$.

Now using the total dimension formuli of Propositions \ref{prop:rankpos} and \ref{prop:rankneg} we conclude
$$
2 = \dim(\HFFr) = q\delta(K),
$$
which is impossible for $q = 3$ or $q = -5$.

Comparing the labelling of \Sps\ we see that the order in which we listed $HF^+(Y,i)$ above corresponds to $i = 0, 1, 2$ and 3.

If $Y$ could be obtained by $(-4)$-surgery on an $L$-space knot, then the fact that $V_0 = 0$ would imply that its genus is zero, i.e. it is the unknot. However, $Y$ is not a lens space.
\end{proof}

It seems worth noticing that in fact there are infinitely many knots $K_n$ that produce $Y$ from the proposition above -- see \cite{teragaito}. In fact, $K_0 = \overline{9_{42}}$. $\frac{p}{q}$-surgeries on these knots have rather similar Floer homologies, in particular, all the correction terms are the same (and coincide with the correction terms of the lens space $L(p,q)$) and the total rank of reduced Floer homology is $2q$.

Moreover, we can work out the Heegaard Floer homology of all surgeries on these knots and their Alexander polynomials. Teragaito mentions in \cite[Remark 6.1]{teragaito} that $K_n$ has genus $2n+2$. In \cite[Corollary 4.5]{OSzKnotInv} it is shown that
$$
\widehat{HFK}(K,g(K)) \cong HF^+(S^3_0(K),g-1)
$$

so it is non-trivial by Theorem \ref{thm:genus_detect} and thus by Proposition \ref{zerosurg} and the fact that for present examples $V_0 = 0$ we get that $\boldsymbol{A}^{red}_{\pm (g(K)-1)}$ have to be non-trivial. By description of the Heegaard Floer homology of $Y$ in the proof of Prooposition \ref{prop:terag} we conclude that $\boldsymbol{A}^{red}_{2n+1}(K_n) = \boldsymbol{A}^{red}_{-(2n+1)}(K_n) = \tau(1)$ and $\boldsymbol{A}^{red}_{k}(K_n) = 0$ for any $k \neq \pm 2n+1$. Using Proposition \ref{prop:hfineg} we can also fix the gradings and then using results from section 3 deduce the Heegaard Floer homology of all surgeries on these knots.

\begin{prop}
The Alexander polynomial of $K_0$ is $-1+2(T+T^{-1})-(T^2+T^{-2})$. For $n \neq 0$ the Alexander polynomial is given by
$$
\Delta_{K_n}(T) = 1 - (T^{2n}+T^{-2n})+2(T^{2n+1}+T^{-(2n+1)})-(T^{2n+2}+T^{-(2n+2)}).
$$
\end{prop}
\begin{proof}
From the discussion above, $V_0 = 0$ and the only non-trivial $\hAr$'s are $\boldsymbol{A}^{red}_{2n+1}(K_n) = \boldsymbol{A}^{red}_{-(2n+1)}(K_n) = \tau(1)$. Moreover, since the reduced parts of the Heegaard Floer homology of $(-4)$-surgery are in absolute $\ZZ/2\ZZ$-grading $0$, it means that $\boldsymbol{A}^{red}_{\pm(2n+1)}$ are in grading $1$. (We can see from the description of the absolute grading on the mapping cone and Lemma \ref{lemma:dtneg} that for negative surgeries the $\ZZ/2\ZZ$-grading of $\hAA$ switches from what we have defined it to be in the mapping cone.)
Now Lemma \ref{lemma:euler} implies that $t_{2n+1} = -1$ and $t_i = 0$ for all other $i \geq 0$.
\end{proof}

By a straightforward argument involving $\ZZ/2\ZZ$-grading considerations and dimension count it is not difficult to establish that in fact for $n>0$ 
$$
\widehat{HFK}(K_n,2n+2) \cong \widehat{HFK}(K_n,2n) \cong \FF \mbox{ and } \widehat{HFK}(K_n,2n+1) \cong \FF^2.
$$ 

\subsection{Property S}

Heegaard Floer homology has been very successful in restricting cosmetic surgeries on knots in $S^3$ (see \cite{niWu}, \cite{OSzRatSurg}, \cite{wangCosmetic}). In this subsection, we define a class of knots that do not admit purely cosmetic surgeries.

\begin{defn}
Let $r_1, r_2 \in \mathbb{Q}$, and $K \subset S^3$ be a knot. The surgeries on $K$ with slopes $r_1$ and $r_2$ are called cosmetic if $S^3_{r_1}(K)$ is homeomorphic to $S^3_{r_2}(K)$. They are called purely cosmetic if $S^3_{r_1}(K) \cong S^3_{r_2}(K)$, by which we mean that there exists an orientation preserving homeomorphism between them.
\end{defn}

We now begin defining the property that will imply the non-existence of purely cosmetic surgeries.

\begin{defn}
We say that a rational homology sphere $Y$ has property S if $HF_{red}(Y)$ is all concentrated in the same absolute $\ZZ/2\ZZ$-grading.
\end{defn}

\begin{defn}
We say that a knot $K \subset S^3$ has property S if $\KS$ has property S for some $p/q \neq 0$.
\end{defn}

\begin{prop}
$K$ has property S if and only if for any $p/q \geq 2g(K)-1, \ \KS$ has property S.
\end{prop}

\begin{proof}
Suppose $\KS$ has property S. By taking the mirror of the knot, we may assume $p/q >0$.

Then by looking at Corollary \ref{cor:dtpos} and Proposition \ref{prop:hfipos} we see that for all  $k$ all elements of $\hAr$ are in the same $\ZZ/2\ZZ$-grading. This is enough for all elements of $\HFF$ for $p/q \geq 2g(K)-1$ to be concentrated in the same $\ZZ/2\ZZ$-grading.
\end{proof}

\begin{cor}
There are no purely cosmetic surgeries on non-trivial knots with Property S.
\end{cor}

\begin{proof}
The proof is completely analogous to \cite[proof of Corollary 3.12]{niWu}. In fact, Ni and Wu show that if $Y$ can be obtained by a purely cosmetic surgery, then the Euler characteristic of $\HFFr$ has to be 0. They also show that $V_0$ and $\overline{V_0}$ have to be zero for a knot that admits cosmetic surgeries. This implies that $V_i = H_i = 0$ for $i \geq 0$, so we do not have any $\tau(V_i)$ groups in the reduced Floer homology. A knot with property S has all the $\hA$ groups concentrated in the same $\ZZ/2\ZZ$-grading and in the case at hand these are the groups that constitute the reduced Floer homology. Therefore, in this case, the Euler characteristic of $\HFFr$ is equal to ($\pm$) its rank, so it is an $L$-space. however, if an $L$-space knot has $V_0 = 0$, then it is trivial.
\end{proof}

In \cite[Corollary 3.12]{niWu} Ni and Wu show that Seifert fibred spaces cannot be obtained by purely cosmetic surgeries. We can extend this result as follows.

\begin{cor}
There are no purely cosmetic surgeries on knots with non-zero Seifert fibred surgeries.
\end{cor}

\begin{proof}
By \cite{OSzPlumbed} Seifert fibred rational homology spheres have property S.
\end{proof}

We remark that there are knots which do not have this property, for example $9_{44}$. Indeed, $+1$ and $-1$-surgeries on this knot have the same $HF^+$, but are not homeomorphic \cite[Section 9]{OSzRatSurg}.

\bibliographystyle{plain}
\bibliography{biblio}

\begin{thebibliography}{10}

\bibitem{crowell}
Richard~H. Crowell.
\newblock Nonalternating links.
\newblock {\em Illinois J. Math.}, 3:101--120, 1959.

\bibitem{greeneCabling}
Joshua~Evan Greene.
\newblock L-space surgeries, genus bounds, and the cabling conjecture.
\newblock {\em J. Differential Geom.}, 100(3):491--506, 2015.

\bibitem{jabukaHat}
S.~Jabuka.
\newblock Heegaard {F}loer groups of {D}ehn surgeries.
\newblock {\em J. Lond. Math. Soc. (2)}, 92(3):499--519, 2015.

\bibitem{jabukaGenus}
Stanislav Jabuka.
\newblock Heegaard {F}loer genus bounds for {D}ehn surgeries on knots.
\newblock {\em J. Topol.}, 7(2):523--542, 2014.

\bibitem{lackenbyDehnSurg}
Marc Lackenby.
\newblock Dehn surgery on knots in {$3$}-manifolds.
\newblock {\em J. Amer. Math. Soc.}, 10(4):835--864, 1997.

\bibitem{lackenbyPurcell}
Marc Lackenby and Jessica Purcell.
\newblock Cusp volumes of alternating knots.
\newblock {\em Geom. Topol.}, 20(4):2053--2078, 2016.

\bibitem{mccoy}
Duncan McCoy.
\newblock Non-integer surgery and branched double covers of alternating knots.
\newblock {\em J. Lond. Math. Soc. (2)}, 92(2):311--337, 2015.

\bibitem{mooreStarkston}
Allison~H. Moore and Laura Starkston.
\newblock Genus-two mutant knots with the same dimension in knot {F}loer and
  {K}hovanov homologies.
\newblock {\em Algebr. Geom. Topol.}, 15(1):43--63, 2015.

\bibitem{murasugiAlexAlt}
Kunio Murasugi.
\newblock On the {A}lexander polynomial of the alternating knot.
\newblock {\em Osaka Math. J.}, 10:181--189; errata, 11 (1959), 95, 1958.

\bibitem{niWu}
Yi~Ni and Zhongtao Wu.
\newblock Cosmetic surgeries on knots in {$S^3$}.
\newblock {\em To appear in J. Reine Angew. Math}, 2013.

\bibitem{niZhangChar}
Yi~Ni and Xingru Zhang.
\newblock Characterizing slopes for torus knots.
\newblock {\em Algebr. Geom. Topol.}, 14(3):1249--1274, 2014.

\bibitem{osoinach}
John~K. Osoinach, Jr.
\newblock Manifolds obtained by surgery on an infinite number of knots in
  {$S^3$}.
\newblock {\em Topology}, 45(4):725--733, 2006.

\bibitem{OSzAbsGr}
Peter Ozsv{\'a}th and Zolt{\'a}n Szab{\'o}.
\newblock Absolutely graded {F}loer homologies and intersection forms for
  four-manifolds with boundary.
\newblock {\em Adv. Math.}, 173(2):179--261, 2003.

\bibitem{OSzPlumbed}
Peter Ozsv{\'a}th and Zolt{\'a}n Szab{\'o}.
\newblock On the {F}loer homology of plumbed three-manifolds.
\newblock {\em Geom. Topol.}, 7:185--224 (electronic), 2003.

\bibitem{OSzGenusBounds}
Peter Ozsv{\'a}th and Zolt{\'a}n Szab{\'o}.
\newblock Holomorphic disks and genus bounds.
\newblock {\em Geom. Topol.}, 8:311--334, 2004.

\bibitem{OSzKnotInv}
Peter Ozsv\'{a}th and Zolt\'{a}n Szab\'{o}.
\newblock {Holomorphic disks and knot invariants}.
\newblock {\em Advances in Mathematics}, 186(1):58--116, August 2004.

\bibitem{OSzPropApp}
Peter Ozsv{\'a}th and Zolt{\'a}n Szab{\'o}.
\newblock Holomorphic disks and three-manifold invariants: properties and
  applications.
\newblock {\em Ann. of Math. (2)}, 159(3):1159--1245, 2004.

\bibitem{OSzTopInv}
Peter Ozsv{\'a}th and Zolt{\'a}n Szab{\'o}.
\newblock Holomorphic disks and topological invariants for closed
  three-manifolds.
\newblock {\em Ann. of Math. (2)}, 159(3):1027--1158, 2004.

\bibitem{OSzSeifFibr}
Peter Ozsv{\'a}th and Zolt{\'a}n Szab{\'o}.
\newblock On {H}eegaard {F}loer homology and {S}eifert fibered surgeries.
\newblock In {\em Proceedings of the {C}asson {F}est}, volume~7 of {\em Geom.
  Topol. Monogr.}, pages 181--203 (electronic). Geom. Topol. Publ., Coventry,
  2004.

\bibitem{OSzLensSpaceSurg}
Peter Ozsv{\'a}th and Zolt{\'a}n Szab{\'o}.
\newblock On knot {F}loer homology and lens space surgeries.
\newblock {\em Topology}, 44(6):1281--1300, 2005.

\bibitem{OSzIntegerSurgeries}
Peter~S. Ozsv{\'a}th and Zolt{\'a}n Szab{\'o}.
\newblock Knot {F}loer homology and integer surgeries.
\newblock {\em Algebr. Geom. Topol.}, 8(1):101--153, 2008.

\bibitem{OSzRatSurg}
Peter~S. Ozsv{\'a}th and Zolt{\'a}n Szab{\'o}.
\newblock Knot {F}loer homology and rational surgeries.
\newblock {\em Algebr. Geom. Topol.}, 11(1):1--68, 2011.

\bibitem{rasmussenGodaTerag}
Jacob Rasmussen.
\newblock Lens space surgeries and a conjecture of {G}oda and {T}eragaito.
\newblock {\em Geom. Topol.}, 8:1013--1031, 2004.

\bibitem{rasmussenThesis}
Jacob~Andrew Rasmussen.
\newblock {\em Floer homology and knot complements}.
\newblock ProQuest LLC, Ann Arbor, MI, 2003.
\newblock Thesis (Ph.D.)--Harvard University.

\bibitem{rolfsen}
Dale Rolfsen.
\newblock {\em Knots and links}, volume~7 of {\em Mathematics Lecture Series}.
\newblock Publish or Perish Inc., Houston, TX, 1990.
\newblock Corrected reprint of the 1976 original.

\bibitem{teragaito}
Masakazu Teragaito.
\newblock A {S}eifert fibered manifold with infinitely many knot-surgery
  descriptions.
\newblock {\em Int. Math. Res. Not. IMRN}, (9):Art. ID rnm 028, 16, 2007.

\bibitem{wangCosmetic}
Jiajun Wang.
\newblock Cosmetic surgeries on genus one knots.
\newblock {\em Algebr. Geom. Topol.}, 6:1491--1517, 2006.

\bibitem{wuCones}
Zhongtao Wu.
\newblock On mapping cones of {S}eifert fibered surgeries.
\newblock {\em J. Topol.}, 5(2):366--376, 2012.

\end{thebibliography}

\end{document}